\documentclass{article}

\usepackage{lipsum}
\usepackage{epstopdf}

\usepackage{graphicx}
\usepackage{amsmath,amsfonts,amsthm}
\usepackage{hyperref, booktabs, subcaption}
\usepackage{times}
\usepackage{algorithm}
\usepackage{algorithmic}
\usepackage{multirow,multicol}
\usepackage{enumerate}
\usepackage[margin=1in]{geometry}

\usepackage{etoolbox}
\newbool{DrawFigures}
 \booltrue{DrawFigures} 

\usepackage{tikz, pgfplots}
\pgfplotsset{compat=newest}
\newlength{\figurewidth}
\newlength{\figureheight}
\usepackage{mwe}
\usepackage[T1]{fontenc}

\newcommand{\R}{\mathbb{R}}

\newcommand{\para}{\parallel}

\newcommand{\diag}{\operatorname{diag}}
\newcommand{\trace}{\operatorname{\textbf{tr}}}

\newcommand{\nn}{\nonumber}
\newcommand{\argmin}{\operatornamewithlimits{arg\ min}}
\newcommand{\argmax}{\operatornamewithlimits{arg\ max}}

\newtheorem{theorem}{Theorem} 
\newtheorem{definition}{Definition} 

\newtheorem{corollary}[theorem]{Corollary}
\newtheorem{lemma}[theorem]{Lemma}
\newtheorem{conjecture}{Conjecture}

\def\ie{{\em i.e.,~}}
\def\eg{{\em e.g.,~}}

\newcommand{\revise}[1]{{\color{black}{#1}}}
\newcommand{\new}[1]{{\color{black}{#1}}}
\newcommand{\newkeep}[1]{{\color{blue}{#1}}}
\newtheorem{assumption}{Assumption}
\begin{document}

\title{Convergence of a Grassmannian Gradient Descent Algorithm for Subspace Estimation From Undersampled Data \thanks{The work of both authors in this publication was supported by the U.S. Army Research Office under grant number W911NF1410634.}
}


\author{Dejiao Zhang and Laura Balzano \\
dejiao \& girasole@umich.edu \\
Department of Electrical Engineering and Computer Science \\
University of Michigan, Ann Arbor }

\date{}

\maketitle

\begin{abstract}
Subspace learning and matrix factorization techniques have many applications in science and engineering, and efficient algorithms are critical as dataset sizes continue to grow. Many relevant problem formulations are non-convex, and in a variety of contexts it has been observed that solving the non-convex problem directly is not only efficient but reliably accurate. We discuss convergence theory for a particular method: first order incremental gradient descent constrained to the Grassmannian. The output of the algorithm is an orthonormal basis for a $d$-dimensional subspace spanned by an input streaming data matrix.  We study two sampling cases: where each data vector of the streaming matrix is fully sampled, or where it is undersampled by a sampling matrix $A_t\in \mathbb{R}^{m\times n}$ with $m\ll n$. Our results cover two cases, where $A_t$ is Gaussian or a subset of rows of the identity matrix. We propose an adaptive stepsize scheme that depends only on the sampled data and algorithm outputs. We prove that with fully sampled data, the stepsize scheme maximizes the improvement of our convergence metric at each iteration, and this method converges from any random initialization to the true subspace, despite the non-convex formulation and orthogonality constraints. For the case of undersampled data, we establish monotonic expected improvement on the defined convergence metric for each iteration with high probability.

\newkeep{This technical report was updated in February 2022 to match Dejiao Zhang's PhD dissertation \cite{zhang2019extracting}, which corrected some errors. For the case with full observations (no compressed or missing data), the theoretical results herein have been superseded by several other results in the literature, including results for the GROUSE algorithm itself \cite{balzano2022equivalence}.}
\end{abstract}

\section{Introduction}
Low-rank matrix factorization is an essential tool for high-dimensional inference with fewer measurements than variables of interest, where low-dimensional models are necessary to perform accurate and stable inference. 
Many modern problems fit this paradigm, where signals are undersampled because of sensor failure, resource constraints, or privacy concerns. 
Suppose we wish to factorize a matrix $M=UW^T$ when we only get a small number of linear measurements of $M$. Solving for the subspace basis $U$ can be computationally burdensome in this undersampled problem and related regularized problems. Many algorithms that attempt to speed up computation are solving a non-convex optimization problem, and therefore come with few guarantees.

The Singular Value Decomposition (SVD) provides the solution to the non-convex matrix factorization problem formulation with full data, and there are several highly successful algorithms for solving it \cite{golub2012matrix}. Unfortunately, these algorithms cannot easily be extended to problems with incomplete observations of the matrix. Recently, several results have been published with first-of-their-kind guarantees for a variety of different gradient-type algorithms on non-convex matrix factorization problems \cite{armentano2014average,bhojanapalli2016dropping,chen2015fast,sa2015global,jain2016streaming,jain2013low,zheng2015convergent}. These new algorithms, being gradient-based, are well-suited to extensions of the SVD where the matrix is not fully sampled and where we include different cost functions or regularizers. For example, with gradient methods to solve the SVD we may be able to solve Robust PCA \cite{candes2011robust,he2012incremental,xu2010robust}, Sparse PCA \cite{d2008optimal}, or even $\ell_1$ PCA \cite{brooks2013pure} with gradient methods as well. However, almost none of these results gives guarantees in \emph{streaming problem}, where data can only be accessed one partial column vector at a time. This is a critical problem in the modern machine learning context with massive data and comparatively limited memory, or in applications where data are collected continuously and must be processed in realtime. The existing theoretical results for the streaming problem significantly overestimate the number of samples needed for convergence for typical algorithms.
 
Our contribution is to provide a global convergence result for $d$-dimensional subspace estimation using an incremental gradient algorithm performed on the Grassmannian, the space of all $d$-dimensional subspaces of $\R^n$, denoted by $\mathcal{G}(n,d)$. 
Subspace estimation is a special case of matrix factorization with orthogonality constraints, where we seek to estimate only the subspace spanned by the columns of the left matrix factor $U \in \R^{n\times d}$. 
\new{Our result demonstrates that, for fully sampled data without noise, this gradient algorithm \emph{converges globally to the global minimizer} almost surely, \ie it converges from any random initialization to the global minimizer. For undersampled data, including compressively sampled data and missing data, we provide results showing monotonic improvement in expectation on the metric of convergence for each iteration.} 


This paper is organized as follows. The problem formulation and the GROUSE algorithm are described in Section \ref{sec:formulation_and_alg}. The global convergence result for fully sampled data is presented in Section \ref{sec:global_full}, the convergence behavior of GROUSE with undersampled data is studied in Section \ref{sec:undersampled_data}, and the corresponding proofs are provided in Sections \ref{sec:appendix_a}, \ref{sec:proof_full_nfree} and \ref{sec:proof_of_undersampled_data}. Experiment results are in Section \ref{sec:experment}.

\section{Problem Setting}
\label{sec:formulation_and_alg}
In this paper, we consider the problem of learning a low dimensional subspace representation from streaming data. Specifically, we are given a sequence of observations $x_t = A_t v_t$ where $A_t \in \R^{m\times n}$ $(m\leq n)$ are sampling matrices that are given for each observation;  $v_t\in \R^{n}$ are drawn from a continuous distribution with support on the true subspace, spanned by $\bar U\in \R^{n\times d}$ with orthonormal columns, \ie $v_t = \bar U s_t, s_t \in \R^d$. In this paper, we study three different sampling frameworks: the fully sampled case with $A_t$ being the identity matrix, the compressively sampled case with $A_t \in \R^{m\times n}$ $(m\ll n) $ being random Gaussian matrices, and the missing data case where each row of $A_t$ $(m\ll n)$ is uniformly sampled from the identity matrix. 

We formulate subspace estimation as a non-convex optimization problem as follows. Let $U\in \mathbb{R}^{n\times d}$ be a matrix with orthonormal columns. Then we want to solve: 
\begin{align}
  \underset{U \in \R^{n\times d}}{\text{minimize}} \quad & \sum_{t = 1}^{T} \min_{w_t} \|A_t U w_t - x_t\|_2^2 \label{eq:obj}\\
   \text{subject to } & \quad  \text{span}\left(U\right)\in  \mathcal{G}(n,d)\nn
\end{align}
This problem is non-convex firstly because of the product of the two variables $U$ and $w_t$ and secondly because the optimization is over the Grassmannian $\mathcal{G}(n,d)$, the non-convex set of all $d$-dimensional subspaces in $\mathbb{R}^{n}$. We study an online algorithm to solve the above problem, where we process one observation at a time and perform a rank-one update to generate a sequence of estimates $U_t$ with the goal that $R(U_t)\rightarrow R(\bar U)$, where $R(\cdot)$ denotes the column range.

We can see the relationship between our problem and the well studied low-rank matrix recovery problem. Let $W\in \R^{d\times T}$ and $M = [v_1, \dots, v_T] \in \R^{n\times T}$, then (\ref{eq:obj}) is equivalent to 
\begin{align}
  \underset{U \in \R^{n\times d}, W\in \R^{d\times T}}{\text{minimize}} \quad & \|\mathcal{A}\left(U W\right) - \mathcal{A}\left(M\right)\|_2^2 \label{eq:objmatch}\\
  \quad \text{subject to } \quad  &\text{span}\left(U\right)\in  \mathcal{G}(n,d)\nn
\end{align}
where $\mathcal{A}: \R^{n\times T} \rightarrow \R^{mT}$ is a linear operator. Our algorithm can be thought of as an incremental algorithm to solve this problem as well. 
Fueled by the great deal of recent success of directly solving non-convex factorization problems (as we discuss in related work below), we study the natural incremental gradient descent algorithm \cite{bertsekas2011incremental} applied to (\ref{eq:obj}) directly. Since the optimization variable in our problem is a subspace, we constrain the gradient descent to the Grassmannian $\mathcal{G}(n,d)$. The resulting algorithm is called GROUSE (Grassmannian Rank-One Update Subspace Estimation) algorithm and is described in Algorithm \ref{alg:grouse}. This description differs from its initial introduction in \cite{balzano2010online} in that it extends the missing data case to a more general sampling framework.

\begin{algorithm}
\caption{GROUSE: Grassmannian Rank-One Update Subspace Estimation} \label{alg:grouse}
\begin{algorithmic}
\STATE{Given $U_0$, an $n \times d$ matrix with orthonormal columns, with $0<d<n$;}
\STATE{Set $t:= 0$;}
\REPEAT 
\STATE{Given sampling matrix $A_t: \R^n \rightarrow \R^m$ and observation $x_t = A_t v_t$;}
\STATE{Define $w_t := \arg \min_a \|A_tU_t a - x_t \|^2$;}
\STATE{Define $p_t := U_t w_t$ and $\widetilde r_t := x_t - A_tp_t$, $r_t := A_t^T\widetilde r_t$;}
\STATE{Using step size
\vspace{-3mm}
\begin{equation}
\theta_t =  \arctan\left(\frac{\|r_t\|}{\|p_t\|}\right)
\label{eq:theta}
\end{equation} 
update with a gradient step on the Grassmannian:}
\begin{equation} \label{eq:gpupdate}
U_{t+1} := U_t + \left(\frac{y_t}{\|y_t\|} - \frac{p_t}{\|p_t\|}\right)\frac{w_t^T}{\|w_t\|}
\end{equation} 

\vspace{-2mm}
\noindent where 
\vspace{-2mm}
$$\frac{y_t}{\|y_t\|_2} = \frac{p_t}{\|p_t\|_2}\cos (\theta_t)  + \frac{r_t}{\|r_t\|_2}\sin(\theta_t)$$
\vspace{-2mm}
\STATE{$t:=t+1$;}
\UNTIL{termination}
\end{algorithmic}
\end{algorithm}

\subsection{Algorithm}
\label{subsec:algorithm}
At each step, the GROUSE algorithm receives a vector $x_t  = A_t v_t$, and tries to minimize the inconsistency between $R(U)$ and the true subspace $R(\bar U)$ with respect to the information revealed in the sampled vector $x_t$, \ie
\begin{equation}
  \mathcal{F}\left(U; t\right) = \min_a\left\|A_tU a - x_t\right\|^2
  \label{eq:subsapce_error}
\end{equation}
In order to do so, GROUSE forms the gradient of $\mathcal{F}$ with respect to $U$ evaluated at the current estimate $U_t$, and takes a step in the direction of the negative gradient restricted to the Grassmannian. The derivation of the incremental gradient descent update rule on the Grassmannian is found in \cite{balzano2010online,balzano2012handling}, and we summarize it here. 

To compute the gradient of $\mathcal{F}$ on the Grassmannian, we first need to compute the derivative of $\mathcal{F}$ with respect to $U$ and evaluate it at $U_t$. As we will prove later, under mild conditions, $A_tU_t$ has full column rank with high probability. Therefore, the derivative is 
\begin{equation}
  \frac{d \mathcal{F}}{d U} = -2 A_t^T \widetilde r_t w_t^T
  \label{eq:subspace_gradient}
\end{equation}
where $\widetilde r := x_t - A_tU_t w_t$ denotes the residual vector with respect to the sampled vector $x_t$, and $w_t$ is the least-squares solution of (\ref{eq:subsapce_error}). 
Using Equation (2.70) in \cite{edelman1998geometry}, the gradient of $\mathcal{F}$ on the Grassmannian then follows as 
\begin{align}
  \nabla \mathcal{F} = \left(I - U_tU_t^T\right)\frac{d \mathcal{F}}{d U} &= -2\left(I - U_tU_t^T\right) A_t^T \widetilde r_t w_t^T \nn \\&= -2 A_t^T \widetilde r_t w_t^T \;.
  \label{eq:tangent_vector}
 \end{align} 
The final equality follows by $\widetilde{r_t} \perp A_t U_t$, which can be verified using the definitions of $w_t$ and $\widetilde r_t$. According to Eq (2.65) in \cite{edelman1998geometry}, a gradient step along the geodesic with tangent vector $-\nabla \mathcal{F}$ can be then formed as a function of the singular values and singular vectors of $\nabla \mathcal{F}$. For this specific case of our rank one $\nabla\mathcal{F}$ given in \eqref{eq:tangent_vector}, the update rule follows as 
\begin{equation}
  U(\eta) = U_t + \left[\left(\cos\left(\eta_t\sigma_t\right)-1\right)\frac{U_tw_t}{\|w_t\|} + \sin\left(\eta_t\sigma_t\right)\frac{A_t^T\widetilde r_t}{\|A_t^T\widetilde r_t\|}\right]\frac{w_t^T}{\|w_t\|}
  \label{eq:gpupdate_0}
\end{equation}
where  $\eta_t>0$ is the chosen step size at iteration $t$, $p_t := U_t w_t$ is the predicted value of the projection of the vector $v_t$ onto $R(U_t)$ and $\sigma_t=\|A_t^T\widetilde r_t\|\|p_t\|$. By leveraging the fact that $\widetilde r_t \perp A_t U_t$ and $p_t \in R(U_t)$, it's easy to verify that the rank-one update (\ref{eq:gpupdate_0}) maintains orthogonality $U(\eta)^T U(\eta)=\mathbb{I}_d$, and tilts $R(U_t)$ to a new point on Grassmannian.

In summary, for each observation the GROUSE algorithm works as follows:  it projects the data vector onto the current estimate of the true subspace with respect to the sampling matrix $A_t$, to get either the exact (when $A_t = \mathbb{I}_n$) or approximated projection $p_t$ and residual $r_t = A_t^T\widetilde r_t$. Then GROUSE updates the current estimate with a rank-one step as described by (\ref{eq:gpupdate}).  In the present work, we propose an adaptive stepsize framework that sets the stepsize only based on the sampled data and the algorithm outputs. More specifically, at each iteration a stepsize $\eta_t$ is chosen such that $\eta_t \sigma_t = \arctan\left(\frac{\|r_t\|}{\|p_t\|}\right)$.  As shown in Section \ref{sec:global_full}, the proposed stepsize scheme is greedy for the fully sampled data, \ie it maximizes the improvement of our defined convergence metric at each iteration. For the undersampled data, we establish a local convergence result by showing that, with the proposed stepsize, GROUSE moves the current estimated subspace towards the true subspace with high probability despite the nonconvex nature of the problem and undersampled data.

\subsection{Related Work}
\label{subsec:related_work}
Many recent results have shown theoretical support for directly solving non-convex matrix factorization problems with gradient or alternating minimization methods. Among the incremental methods \cite{sa2015global} is the one closest to ours, where the authors consider recovering a positive semidefinite matrix with undersampled data. They propose a step size scheme with which they prove global convergence results from a randomly generated initialization. However, their convergence results contain a obscure term, and their choice of step size depends on the knowledge of some parameters that are likely to be unknown in practical problems. Without this knowledge, the results only hold with sufficiently small step size that implies significantly slower convergence. 

\new{In contrast, while our work applies more narrowly to the subspace estimation problem, we provide an explicit expression for the expected improvement at each iteration, using a step size that only depends on the observations and outputs of the algorithms. Based on that, we prove that with fully sampled data, the proposed stepsize scheme maximizes the improvement of our convergence metric at each iteration, and GROUSE converges from any random initialization to the true subspace, despite the non-convex formulation and orthogonality constraint global convergence. We further posit a conjecture on the global convergence rate that better matches the practical observations for fully sampled data. Although we have not yet established a complete proof of this conjecture, we present our current approach in Appendix \ref{sec:proof_full_nfree}.} 

Other work that has looked at incremental methods has focused only on fully sampled vectors. For example, \cite{balsubramani2013fast} invokes a martingale-based argument to derive the global convergence rate of the proposed incremental PCA method to the single top eigenvector in the fully sampled case. In contrast, \cite{arora2013stochastic} estimates the best $d$-dimensional subspace in the fully sampled case and provides a global convergence result by relaxing the non-convex problem to a convex one. We seek to identify the $d$ dimensional subspace by solving the non-convex problem directly.

The results in this paper are very closely related to our previous work~\cite{balzano2014local}. In \cite{balzano2014local}, we prove that, within a local region of the true subspace, an expected improvement of their defined convergence metric for each iteration of GROUSE can be obtained. In contrast, we establish global convergence results to a global minimizer from any random initialization for fully sampled data, and extend the local convergence results to compressively sampled data. We also expand the local convergence results in \cite{balzano2014local} to a much less conservative region, and we provide a much simpler analysis framework that can be applied to different sampling strategies. Moreover, for each iteration of the GROUSE algorithm, the expected improvement on the convergence metric defined in \cite{balzano2014local} only holds locally in both theory and practice, while our theoretical result provides a tighter bound for the global convergence behavior of GROUSE over a variety of simulations. This suggests that our result has more promise to be extended to a global result for both missing data and compressively sampled data. 

Turning to batch methods, \cite{RH2012,jain2013low} provided the first theoretical guarantee for an alternating minimization algorithm for low-rank matrix recovery in the undersampled case. Under typical assumptions required for the matrix recovery problems \cite{recht2010guaranteed}, they established geometric convergence to the global optimal solution. Earlier work \cite{keshavan2010matrix,ngo2012scaled} considered the same undersampled problem formulation and established convergence guarantees for a steepest descent method (and a preconditioned version) on the full gradient, performed on the Grassmannian. {\cite{chen2015fast,bhojanapalli2016dropping,zheng2015convergent} considered low rank semidefinite matrix estimation problems, where they reparamterized the underlying matrix as $M = UU^T$, and update $U$ via a first order gradient descent method. However, all these results require batch processing and a decent initialization that is close enough to the optimal point, resulting in a heavy computational burden and precluding problems with streaming data.}
We study random initialization, and our algorithm has fast, computationally efficient updates that can be performed in an online context. 

Lastly, several convergence results for optimization on general Riemannian manifolds, including several special cases for the Grassmannian, can be found in \cite{absil2009optimization}. Most of the results are very general; they include global convergence rates to local optima for steepest descent, conjugate gradient, and trust region methods, to name a few. We instead focus on solving the problem in \eqref{eq:obj} and provide global convergence rates to the global minimum.

Before we present the main results, we first call out the following notation which we use throughout this chapter.  For notational convenience, we will drop the iteration subscript except our convergence metric $\zeta_t$ defined in Definition \ref{defn:detdiscrepancy} hereafter.

\paragraph{\textbf{Notation}}{We use $R(M)$ to denote the column space of a matrix $M$ and $\mathcal{P}_M$ to denote the orthogonal projection onto $R(M)$. $\mathbb{I}_n$ denotes the identity matrix in $\R^{n\times n}$ and $M_i$ denotes the $i^{th}$ row of matrix $M$. In this paper, without specification, $\|\cdot\|$ denotes the $\ell_2$ norm. $R(\bar{U}) $ and $R(U)$ denote the true subspace and our estimated subspace respectively, here both $\bar{U}$ and $U$ are matrices in $\R^{n\times d}$ with orthonormal columns. Also we use $v_{\parallel}$ and $v_\perp$ to denote the projection and residual of the underlying full vector $v\in \R^n$ onto the estimated subspace $R(U)$, \ie $v_\para = UU^T v, v_\perp = v - v_\para$. Note that these two quantities are in general unknown for the undersampled data case. We define them so as to relate the intermediate quantities, determined by the algorithm and sampled data, to the improvement on our defined convergence metric. }

\section{Preliminaries} 
\label{sec:convergence_analysis}
In this section, we first define our convergence metric and describe an assumption on the streaming data needed to establish our results.  Subsequently, we state a fundamental result that is essential to quantify the improvement on the convergence metric over GROUSE iterates. 
\begin{definition}
[Determinant similarity]
Our measure of similarity between $R(U)$ and $R(\bar U)$ is $\zeta \in [0,1]$, defined as
\vspace{-0.3cm}
\begin{equation}
\zeta := \det(\bar U^T U U^T \bar U)  = \prod_{k=1}^d \cos^2 \phi_{k}\;. \nn
\end{equation}
where $\phi_{k}$ denotes the $k^{\text{th}}$ principal angle between $R(\bar U)$ and $R(U)$, where $0 \leq \phi_1 \leq \cdots \leq \phi_d \leq \pi/2$ are defined by $\cos \phi_{k} = \sigma_{k}(\bar U^T U)$ with $\sigma_{k}$ denoting the $k^{th}$ singular value of $\bar U^TU$ (See \cite[Section 6.4.3]{golub2012matrix}).
\label{defn:detdiscrepancy}
\end{definition}

The convergence metric $\zeta$ increases to one when our estimate $R(U)$ converges to $R(\bar U)$, \ie all principal angles between the two subspaces equal zero. Compared to other convergence metrics defined either as $\|(I-\bar{U}\bar{U}^T)U\|_F^2 = d-\|\bar{U}^TU\|_F^2 = \sum_{k=1}^d \sin^2 \phi_{k}$ or $1-\|\bar{U}^TU\|_2^2 = \sin^2 \phi_{1}$, our convergence metric $\zeta$ measures the similarity instead of the discrepancy between $R(U)$ and $R(\bar U)$. In other words, $\zeta$ achieves its maximum value one when $R(U)$ converges to $R(\bar U)$, while the typical subspace distance is zero when the subspaces are equal. Also note that $\zeta=0$ iff at least one of the principal angles is a right angle. That is, all stationary points $U_{stat}$ of the full data problem  except the true subspace have $\det\left(\bar U^TU_{stat}U_{stat}^T\bar U\right) = 0$ \cite{yang1995projection,balzano2012handling}.
\revise{\begin{assumption}
For the underlying data $v = \bar U s$, we assume the entries of $s$ are independent, and identically distributed symmetrically about zero,  and each entry has zero-mean and unit variance. 
  \label{cond:v}
\end{assumption}}

Given this assumption, we have the following lemma which relates the projection $v_\para$ and the projection residual $v_\perp$ to the improvement on our convergence metric $\zeta_t$. As we will show in the following sections, this lemma is crucial for us to establish the expected improvement on our defined convergence metric $\zeta_t$ for all the sampling frameworks considered in this work. The proof is provided in Section \ref{sec:appendix_a}.
\begin{lemma} Let $v_\para$ and $v_\perp$ denote the projection and residual of the full data sample $v$ onto the current estimate $R(U)$. 
Then given Assumption \ref{cond:v}, for each iteration of GROUSE we have
\begin{equation}
  \mathbb{E}\left[\frac{\|v_\perp\|^2}{\|v_\para\|^2} \bigg\lvert U\right] \geq \mathbb{E}\left[\frac{\|v_\perp\|^2}{\|v\|^2}\bigg\lvert U\right] \geq \frac{1 - \zeta_t}{d} \;.
\end{equation} 
  \label{lem:key_quantity}
\end{lemma}
\vspace{-0.2cm}
Although both projection ($v_\para$) and projection residual ($v_\perp$) are in general unknown for the undersampled data, we can relate the approximated projection residual $A^T\widetilde r$ to the true one $v_\perp$ by leveraging either random matrix theory or the incoherence property of the underlying subspace $R(\bar U)$. Therefore, the above lemma provides a unifying step to quantify the improvement on the convergence metric for all cases considered in the present work.

\section{Fully Sampled Data}
\label{sec:global_full} 
In this section, we consider fully sampled data, \ie $A = \mathbb{I}_n$. The corresponding proofs for these results can be found in Section \ref{sec:proof_full_nfree}. We start by deriving a greedy step size scheme for each iteration $t$ that maximizes the improvement on our convergence metric $\zeta_t$. For each update we prove the following:
\begin{equation}
  \frac{\zeta_{t + 1}}{\zeta_t} = \left(\cos\theta + \frac{\|v_{\perp}\|}{\|v_{\parallel}\|}\sin\theta\right)^2 .
  \label{eq:det_ratio_expr}
\end{equation} 
It then follows that 
\begin{equation}
  \theta^{\ast} = \argmax_{\theta} \frac{\zeta_{t + 1}}{\zeta_t} = \arctan \left(\frac{\|v_{\perp}\|}{\|v_{\parallel}\|}\right). 
  \label{eq:stepsize_greedy}
\end{equation}
This is equivalent to (\ref{eq:theta}) in the fully sampled setting $A_t = \mathbb{I}_n$. Using $\theta^\ast$, we obtain monotonic improvement on the determinant similarity that can be quantified by the following lemma. 
\begin{lemma}[Monotonicity for the fully sampled noiseless case]
For fully sampled data, choosing step size $\theta^\ast = \arctan \left(\frac{\|v_{\perp}\|}{\|v_{\parallel}\|}\right)$, after one iteration of GROUSE we obtain
\begin{equation}
    \frac{\zeta_{t + 1}}{\zeta_t} = 1 + \frac{\|v_{\perp}\|^2}{\|v_{\parallel}\|^2} \geq 1 \;.\nn
\end{equation}
  \label{lem:mono_full}
\end{lemma}
\revise{To gain more insight into the improvement on $\zeta_t$ for each iteration of GROUSE, we call out the following lemma, which is a natural result of Lemma \ref{lem:key_quantity} and Lemma \ref{lem:mono_full}.}
\begin{lemma}[Expected improvement on $\zeta_t$] 
When fully sampled data satisfying Assumption \ref{cond:v} are input to the GROUSE (Algorithm \ref{alg:grouse}), 
the expected improvement after one update step is given as: 
  \begin{equation}
    \mathbb{E}\left[\zeta_{t + 1}\big\lvert U\right] \geq \left(1 + \frac{1 - \zeta_t}{d}\right)\zeta_t \;.\nn
  \end{equation}
  \label{lem:exp_mono_zeta} 
\end{lemma}
Under the mild assumption that each data vector is randomly sampled from the underlying subspace, we obtain strict improvement on $\zeta_t$ for each iteration provided $\|v_{\perp}\|>0$ and $\|v_{\parallel}\| >0$. Therefore, Lemma \ref{lem:mono_full} provides insight into how the GROUSE algorithm converges to the global minimum of a non-convex problem formulation: GROUSE is not attracted to stationary points that are not the global minimum. As we mentioned previously, all other stationary points $U_{stat}$ have $\det(\bar U^T U_{stat} U_{stat}^T \bar U) = 0$, because they have at least one direction orthogonal to $\bar U$~\cite{balzano2012handling}. 
Therefore, if the initial point $U_0$ has determinant similarity with $\bar{U}$ strictly greater than zero, then we are guaranteed to stay away from other stationary points, since GROUSE increases the determinant similarity monotonically, according to Lemma \ref{lem:mono_full}. \revise{This together with Lemma \ref{lem:exp_mono_zeta} yields the following convergence result of GROUSE.}

\revise{\begin{theorem}[Convergence of GROUSE] Initialize the starting point $U_0$ of GROUSE such that $\zeta_0>0$. 
Let $1 \geq \zeta^*\geq \zeta_0$ be the desired accuracy of our estimated subspace.  Then for any $\rho>0$, after 
\begin{align}
K &\geq \left(\frac{d}{\zeta_0} + 1\right)\log\left(\frac{1}{\rho(1-\zeta^\ast)}\right) \nn
\end{align} 
iterations of GROUSE Algorithm~\ref{alg:grouse},  
    \begin{equation}
 \mathbb{P}\left(\zeta_K \geq \zeta^\ast\right) \geq 1 - \rho\;. \nn
    \end{equation}.
\label{thm:asymp_full}
\end{theorem}
Notice that if we initialize GROUSE with $U_0$ drawn uniformly from the Grassmannian, \eg as the orthonormal basis of a random matrix $V\in R^{n\times d}$ with entries being independent standard Gaussian variables, this guarantees $\zeta_0 > 0$ with probability one. Therefore, Theorem \ref{thm:asymp_full} provides a global convergence result of GROUSE despite the non-convexity of our objective. However, with this randomly initialized $U_0$, the value of the associated determinant similarity $\zeta_0$ is $\mathcal{O}\left(\left(\frac{d}{n}\right)^d\right)$. Thereby, GROUSE requires $\mathcal{O}\left(d\left(\frac{n}{d}\right)^d\right)$ iterations to converge to the required precision, which is quite pessimistic compared to the actual number of iterations required by GROUSE in numerical simulations. To narrow this gap, we call out the following conjecture on the global convergence rate for GROUSE.}

\begin{conjecture}[Global Convergence of GROUSE]
 Let $1 \geq \zeta^*>0$ be the desired accuracy of our estimated subspace. With the initialization ($U_0$) of GROUSE as the range of an $n\times d$ matrix with entries being i.i.d standard normal random variables, then for any $\rho>0$, after 
\begin{align}
K &\geq K_1 + K_2 \nn \\
&=\left(\frac{2 d^2}{\rho} + 1\right)\tau_0 \log (n) + 2 d \log \left( \frac{1}{2\rho(1 - \zeta^{\ast})}\right) \nn
\end{align} 
iterations of GROUSE Algorithm~\ref{alg:grouse},  
    \begin{equation}
 \mathbb{P}\left(\zeta_K \geq \zeta^\ast\right) \geq 1 - 2\rho\;, \nn
    \end{equation}
    where $\tau_0 = 1 + \frac{\log \frac{(1 - \rho/2)}{C} + d\log (e/d)}{d\log n}$ with $C$ be a constant approximately equal to $1$. 
  \label{thm:global}
\end{conjecture}

\new{This conjecture matches what we see in experimental results. We present a related theorem with additional assumptions in Section \ref{sec:proof_full_nfree}.  We show that the iteration complexity can potentially be a combination of iterations required by two phases: $K_1 = \left(\frac{2 d^2}{\rho} + 1\right)\tau_0 \log (n)$ is the number of iterations required by GROUSE to achieve $\zeta_t \geq 1/2$ from a random initialization $U_0$; and $K_2 = 2 d \log \left( \frac{1}{2\rho(1 - \zeta^{\ast})}\right)$ is the number of additional iterations required by GROUSE to converge to the given accuracy $\zeta^\ast$ from $\zeta_{K_1} = 1/2$. }

We want to comment that conjecture \ref{thm:global} requires fully observed noiseless data, which is not very practical in many cases. However, \new{it would potentially be} the first convergence guarantee for the Grassmannian gradient descent based method for subspace estimation with streaming data. It is a very important initial step for further studies on more general cases, including undersampled data and noisy data with outliers. In the following section, we will analyze the convergence behavior of GROUSE for undersampled data. We leave the corrupted data case as future work.

\section{Undersampled Data} 
\label{sec:undersampled_data}
In this section, we consider undersampled data where each vector $v$ is subsampled by a sampling matrix $A\in \R^{m\times n}$ with the number of measurements being much smaller than the ambient dimension $(m \ll n)$. We study two typical cases, the compressively sampled data where $A$ are random Gaussian matrices, and the missing data where each row of $A$ is uniformly sampled from the identity matrix, $\mathbb{I}_n \in \R^{n\times n}$. 

We first outline several elementary facts that can help us understand how the GROUSE algorithm navigates on the Grassmannian with undersampled data. The proofs can be found in Section \ref{sec:proof_of_undersampled_data}.

Suppose $AU$ has full column rank, then the projection coefficients $w$ are found by the squares solution of $w =: \argmin_a \left\|AUa - x\right\|^2$, \ie $w = (U^TA^TAU)^{-1}U^TA^T x$.
Note that $x = A v$, therefore we can further decompose the projection coefficients $w$ as $w = w_\para + w_\perp$ where
  \begin{equation}
    w_\para = \left(U^TA^TAU\right)^{-1}U^TA^TA v_{\parallel}\;,  \qquad  w_\perp = \left(U^TA^TAU\right)^{-1}U^TA^TA v_{\perp} \;.
    \label{eqs:w_decompose}
  \end{equation}
This decomposition explicitly shows the perturbation induced by the undersampling framework, \ie $Av_\perp$ is not perpendicular to $AU$ in general, though $v_\perp$ is orthogonal to $R(U)$. Now we are going to 
use this perturbation to show how the approximated projection $p$ and residual $r$ deviate
from the exact ones obtained by projecting the full data sample $v$ onto the current estimate $R(U)$.  
\begin{lemma}
  Given Eq \eqref{eqs:w_decompose}, let $p = p_\para + p_\perp$ with $p_\para=Uw_\para$ and $p_\perp = Uw_\perp$, then
  \begin{align}
    p_\para =v_\para  \quad \text{and}\quad r = A^TAv_\perp - A^T \mathcal{P}_{AU}(Av_\perp) \;. 
  \end{align} 
  \label{lem:uniq_expr}
\end{lemma}
\begin{proof}
   Let $a = U^T v_{\parallel}$, then $a$ is the unique solution to $U w = v_{\para}$ given that $U$ has full column rank. Since $AU$ also has full column rank, $b = \left(U^TA^TAU\right)^{-1}U^TA^TAv_{\para}$ is also the unique solution to $AU w = A v_\para$. It then follows that $AUa = A v_\para = AUb$. Therefore, $a = b$. 
   As for the second statement, it simply follows due to the fact that $Av_\para = AUw_\para\in R(AU)$. Hence $\widetilde r = \left(\mathbb{I}_m - \mathcal{P}_{AU}\right)Av = \left(\mathbb{I}_m - \mathcal{P}_{AU}\right)Av_\perp$, recall that $\mathcal{P}_{AU}$ denotes the orthogonal projection operator onto the column space of $AU$. This together with $r = A^T\widetilde r$ completes the proof.
 \end{proof} 
Below we lower bound the improvement on $\zeta_t$ as a function of the key quantities $r, \widetilde r$ and $p$. Compared to Lemma \ref{lem:mono_full}, Lemma \ref{lem:uniq_expr} and Lemma \ref{lem:det_incr_expr_undersample} highlight the how the perturbations induced by the undersampling framework influence the improvement on $\zeta_t$ for each iteration.
Being able to analyze and bound the quantities that include the perturbations is the key to establish the expected improvement on $\zeta_t$ for undersampled data.
\begin{lemma}
Suppose $AU$ has full column rank,  then for each iteration of GROUSE we have
\begin{equation}
  \frac{\zeta_{t + 1}}{\zeta_t} \geq 1 + \frac{2\left\|\widetilde r\right\|^2 - \|r\|^2}{\|p\|^2} + 2\frac{\Delta}{\|p\|^2} 
  \label{eq:det_ratio_undersample}
\end{equation}
where $\Delta  = w_\perp^T\left(\bar U^TU\right)^{-1}\bar U^T r$ with $ w_\perp = \left(U^TA^TAU\right)^{-1}U^TA^TA v_{\perp}$.
  \label{lem:det_incr_expr_undersample}
\end{lemma}
\new{The above lemma highlights the main hurdle in establishing global convergence for undersampled data. As is indicated by (\ref{eq:det_ratio_undersample}), there is no guarantee on monotonicity of the improvement on $\zeta_t$. Indeed, the uncertainty and perturbations introduced by the undersampling framework can even prevent us from establishing monotonically expected improvement on $\zeta_t$. However, we are still able to bound the key quantities in Lemma \ref{lem:det_incr_expr_undersample} and provide more insights on the convergence behavior of GROUSE for both compressively sampled data and missing data.}


\subsection{Compressively Sampled Data} 
\label{sec:compressively_sampled_data}
This section presents convergence results for compressively sampled data. We use an approach that merges linear algebra with random matrix theory to establish an expected rate of improvement on the determinant similarity $\zeta_t$ at each iteration. We show that, under mild conditions, the determinant similarity increases in expectation with a rate similar to that of the fully sampled case, roughly scaled by $\frac{m}{n}$. Detailed proofs for this section are provided in Section \ref{sec:proof_of_undersampled_data}.
\begin{theorem}
  Suppose each sampling matrix $A$ has i.i.d Gaussian entries distributed as $\mathcal{N}(0, 1/n)$. Let $\delta>0$ and let $\phi_d$ denote the largest principal angle between $R(U)$ and $R(\bar U)$. Then with probability exceeding $1 - \exp\left(-\frac{d\delta^2}{8}\right) - \exp\left(-\frac{m\delta^2}{32} + d\log\left(\frac{24}{\delta}\right)\right) - (4d + 2)\exp\left(-\frac{m\delta^2}{8}\right)$ we obtain
\vspace{-0.3cm}
    \begin{equation}
    \mathbb{E}_{v}\left[\zeta_{t + 1} \big\lvert U\right] \geq \left(1 + \gamma_1\left(1 - \gamma_2\frac{d}{m}\right)\frac{m}{n}\frac{1 - \zeta_t}{d}\right)\zeta_t \;, \nn 
  \end{equation}
  where $\gamma_1 = \frac{(1 - \delta)\left(1 - 2\delta\sqrt{\frac{m}{n}}\right)}{\left(1 + \sqrt{\frac{1 + \delta}{1 - \delta}\frac{d}{m}}\right)^2}$ and $\gamma_2  = \left(1 + \frac{2\tan(\phi_d) + \delta\frac{d}{\cos(\phi_d)}}{\left(1 - 2\delta\sqrt{\frac{m}{n}}\right)\sqrt{(1 + \delta)d/m}}\right)\frac{1 + \delta}{1 - \delta}$. Now let $\beta =\frac{8(1+\delta)}{(1-\delta)^2\left(1 - 2\delta\right)^2}$, further suppose 
  {\begin{align}
      m \geq d \cdot \max\left\{\frac{32}{\delta^2}\log\left(\frac{24 n^{2/d}}{\delta}\right), \beta \left(\tan \phi_d + \delta\cos\phi_d d\right)\left(\tan \phi_d + \delta\cos\phi_d d + \frac{1}{2} \right) \right\} \;, \nn
    \end{align}}
  then with probability at least $1 - 2/n^2 - \exp\left(-d\delta^2/8\right)$ we have
  \begin{equation}
    \mathbb{E}_{v}\left[\zeta_{t+1}\big\lvert U\right] \geq \left(1 + \frac{1}{2\gamma_1}\frac{m}{n}\frac{1 - \zeta_t}{d}\right)\zeta_t \;. \nn 
  \end{equation}
  \label{thm:det_convg_csrate}
\end{theorem}
This theorem implies that, for each iteration of GROUSE, expected improvement on $\zeta_t$ can be obtained with high probability as long as the number of samples is enough. As shown in Theorem \ref{thm:det_convg_csrate}, our theory for GROUSE requires more measurements when $R(U)$ is far away from $R(\bar U)$, in which case $\cos\phi_d =: \varepsilon$ is very small. In the high dimensional setting where $m\ll n$, compared to the fully sampled data case, the expected improvement on $\zeta_t$ is approximately scaled down by $\frac{m}{n}$. 
As we will show, this scaling factor is mainly determined by the relative amount of effective information stored in the approximated projection residual. On the other hand, due to the perturbation and uncertainty induced by the compressed sampling framework, the improvement on the determinant similarity given by the lower bound in Lemma \ref{lem:det_incr_expr_undersample} is neither monotonic nor global. As mentioned before, this is the main hurdle to pass before we can provide a global convergence result for undersampled data. However, despite of these difficulties, we are still able to establish Theorem \ref{thm:det_convg_csrate} which shows that, with reasonable number of measurements, the expected improvement on the convergence metric is monotonic with high probability as long as our estimate $R(U)$ is not too far away from the true subspace $R(\bar U)$. 

To prove Theorem \ref{thm:det_convg_csrate}, we provide the following intermediate results to quantify the key quantities in Lemma \ref{lem:det_incr_expr_undersample} with high probability, where probability is taken with respect to the random Gaussian sampling matrix $A$.  

\begin{lemma} 
Under the same conditions as Theorem \ref{thm:det_convg_csrate}, with probability at least $1 - \exp\left(-\frac{m\delta_2^2}{2}\right) - \exp\left(-\frac{m\delta_1^2}{8}\right) - \exp\left(-\frac{d\delta_1^2}{8}\right) $ we obtain
  \begin{align}
   \|\widetilde r\|_2^2 &\geq (1 - \delta_1)\left(1 - \beta\frac{d}{m}\right)\frac{m}{n}\|v_\perp\|_2^2 \label{eq:cs_resid}\\
   2\|\widetilde r\|_2^2 - \|r\|_2^2 
    &\geq (1 - \delta_1)\left(1 - 2\delta_2\sqrt{\frac{m}{n}}\right)\left(1 - \beta\frac{d}{m}\right)\frac{m}{n}\|v_\perp\|_2^2 \label{eq:cs_resid_diff}
  \end{align}
where $\delta_1, \delta_2 \in (0, 1)$, and $\beta = \frac{1 + \delta_1}{1 - \delta_1}$. 
\label{lem:rconc_cs}
\end{lemma}
To interpret the above results, note that 
\begin{align}
  \|\widetilde r\|_2^2 = \left\|\left(\mathbb{I}_m - \mathcal{P}_{AU}\right)Av_\perp\right\|_2^2 = \|Av_\perp\|_2^2 - \left\|\mathcal{P}_{AU}(Av_\perp)\right\|_2^2 \;.
\end{align}
where the first equality follows by the fact that $\left(\mathbb{I}_m - \mathcal{P}_{AU}\right)Av_\para = 0$ as we argued before, and the second equality holds since $\mathcal{P}_{AU}$ is an orthogonal projection onto $R(AU)$. Then by leveraging the concentration property of random projection, we can prove that $\|\widetilde r\|_2^2$ concentrates around its expectation $\frac{m - d}{n}\|v_\perp\|_2^2$ with high probability. Also note that $\|r\|_2^2 \leq \|A\|_2^2 \|\widetilde r\|_2^2$, hence the second statement \eqref{eq:cs_resid_diff} can be established by the concentration result of $\|\widetilde r\|_2^2$ and that of $\|A\|_2^2$ according to the random matrix theory.

Next we establish high probability bounds on $\|p\|_2^2$ and $\Delta$. Then Theorem \ref{thm:det_convg_csrate} follows naturally by first replacing the key quantities in Lemma \ref{lem:det_incr_expr_undersample} with their high probability bounds, and then taking the expectation over the uncertainty of the underlying full data $v_t$.

\begin{lemma}
With the same conditions as Theorem \ref{thm:det_convg_csrate}, for any $\delta_1 \in (0,1)$, we have   
  \begin{equation}
   \|p\|^2 \leq \left(1 + \sqrt{\frac{1 + \delta_1}{1 - \delta_1}\frac{d}{m}}\right)^2\|v\|^2 \nn 
  \end{equation}
with probability at least $1 - \exp\left(-\frac{d\delta_1^2}{8}\right) - \exp\left(-\frac{m\delta_1^2}{32} + d\log\left(\frac{24}{\delta_1}\right)\right)$. 
\label{lem:cs_rpconc}
\end{lemma}

\begin{lemma}
With the same conditions as Theorem \ref{thm:det_convg_csrate}, let $\delta_1, \delta_3 \in (0,1)$, then 
  \begin{equation}
    \Delta \leq \sqrt{\frac{1 + \delta_1}{1 - \delta_1}\frac{d}{m}}\left(\tan(\phi_d) + \delta_3\frac{d}{\cos(\phi_d)}\right)\frac{m}{n}\|v_\perp\|^2   \nn
  \end{equation}
  holds with probability at least $1 - \exp\left(-\frac{d\delta_1^2}{8}\right) - \exp\left(-\frac{m\delta_1^2}{32} + d\log\left(\frac{24}{\delta_1}\right)\right) - 4d\exp\left(-\frac{m\delta_3^2}{8}\right)$.
  \label{lem:exp_delta_cs}
\end{lemma}
Lemma \ref{lem:cs_rpconc} shows that $\|p\|_2^2$ doesn't diverge significantly from $\|v\|_2^2$ as long as $m\geq d$. This together with Lemma \ref{lem:det_incr_expr_undersample} and Lemma \ref{lem:rconc_cs} imply that the required number of measurements in Theorem \ref{thm:det_convg_csrate} is mainly determined by that required by Lemma \ref{lem:exp_delta_cs} so as to prevent $\Delta$ diverging too far from $\frac{m}{n}\|v_\perp\|_2^2$. As a result, the improvement on the determinant similarity is still dominated by the magnitude of the projection residual over that of the projection, which is proportional to that of the full data case scaled by the sampling density. On the other hand, Lemma \ref{lem:exp_delta_cs} implies that, in order to guarantee $\Delta$ to be much smaller than $\frac{m}{n}\|v_\perp\|_2^2$, the number of required measurements increases along with first principal angle between the estimated subspace $R(U)$ and the true subspace $R(\bar U)$.

For the sake of completeness, we sketch the proof of Theorem \ref{thm:det_convg_csrate} here, and the detailed proof is provided in Section \ref{sec:proof_of_undersampled_data}.
\begin{proof}[Proof sketch of Theorem \ref{thm:det_convg_csrate}]
   Let $\eta_1 = \frac{1 + \delta}{1 - \delta}\frac{d}{m}$, $\eta_2 = (1 - \delta)\left(1 - 2\delta\sqrt{\frac{m}{n}}\right)$ and $\eta_3 = \tan(\phi_d) + \delta\frac{d}{\cos(\phi_d)}$,  then plugging in the results in Lemmas \ref{lem:rconc_cs}, \ref{lem:cs_rpconc} and \ref{lem:exp_delta_cs} into Lemma \ref{lem:det_incr_expr_undersample} with $\delta_1=\delta_2=\delta_3=\delta$ yields,
  {\begin{align}
        \frac{\zeta_{t + 1}}{\zeta_t} &\geq 1 + \gamma_1\left(1 - \gamma_2\frac{d}{m}\right)\frac{m}{n}\frac{\|v_\perp\|^2}{\|v\|^2} \geq 1 + \gamma_1\left(1 - \gamma_2\frac{d}{m}\right)\frac{m}{n}\frac{1 - \zeta_t}{d} 
        \label{eq:det_incr_cs_sketch}
    \end{align}}
  where { $\gamma_1 = \frac{(1 - \delta)\left(1 - 2\delta\sqrt{\frac{m}{n}}\right)}{\left(1 + \sqrt{\frac{1 + \delta}{1 - \delta}\frac{d}{m}}\right)^2}$ and $\gamma_2 = \left(1 + 2\frac{\tan(\phi_d) + \delta_3\frac{d}{\cos(\phi_d)}}{\left(1 - 2\delta\sqrt{\frac{m}{n}}\right)\sqrt{(1 - \delta^2)d/m}}\right)\frac{1 + \delta}{1 - \delta}$}. 

\noindent The first probability bound is obtained by taking the union bound of those quantities used to generate Lemma \ref{lem:rconc_cs} to Lemma \ref{lem:exp_delta_cs}, which can be lower bounded by
  {\begin{align}
         1 - \exp\left(-\frac{d\delta^2}{8}\right) - \exp\left(-\frac{m\delta^2}{32}+ d\log\left(\frac{24}{\delta}\right)\right) - (4d + 2)\exp\left(-\frac{m\delta^2}{8}\right) 
        \label{eq:cs_thm_prob_bnd_sketch}
    \end{align}}

\noindent Next we establish the complexity bound on $m$. As we will prove in Section \ref{sec:proof_of_undersampled_data}, $\gamma_2\frac{d}{m}<\frac{1}{2}$ is equivalent to the following,
  {\begin{align}
       m \geq \frac{8(1+\delta)}{(1-\delta)^2\left(1 - 2\delta\right)^2} \left(\varepsilon + \delta\sqrt{1 + \varepsilon^2}d\right)\left(\varepsilon + \delta\sqrt{1 + \varepsilon^2}d + \frac{1}{2} \right) d
      \label{m_bnd_1_sketch}
    \end{align}}
To establish another bound on $m$, $m\geq \frac{32}{\delta^2}\log\left(\frac{24 n^{2/d}}{\delta}\right)d$ implies the following, 
{\begin{align}
  &\exp\left(-\frac{m\delta^2}{32} + d\log\left(\frac{24}{\delta}\right)\right) \leq \exp(-\log{n^2}) = \frac{1}{n^2} \label{prob_bnd_1_sketch} \\
  &(4d + 2)\exp\left(-\frac{m\delta^2}{8}\right) \leq \frac{(4d+2)}{n^8}\left(\frac{\delta}{24}\right)^{4d} \ll \frac{1}{n^2} \label{prob_bnd_2_sketch}
\end{align}}
\eqref{prob_bnd_1_sketch} and \eqref{prob_bnd_2_sketch} complete the proof for the bound on $m$ and justify the simplification of the probability bound in \eqref{eq:cs_thm_prob_bnd_sketch}.
\end{proof}

\subsection{Missing Data} 
\label{sec:missing_data}
In this section, we study the convergence of GROUSE for the missing data case. We show that within the local region of the true subspace, we obtain an expected monotonic improvement on our defined convergence metric with high probability. We use $\Omega$ to denote the indices of observed entries for each data vector, and we assume $\Omega$ is uniformly sampled over $\{1,2,\dots,n\}$ with replacement. In other words, we assume each row of the sampling matrices $A$ is uniformly sampled from the rows of identity matrix $\mathbb{I}_n$ with replacement. We use the notation $A v =: v_{\Omega}, AU =: U_{\Omega}$. Again our results are with high probability with respect to $A$, in this case with respect to the random draw of rows of $\mathbb{I}_n$, and in expectation with respect to the random data $v$. Please refer to Section \ref{sec:proof_of_undersampled_data} for the proofs of this section.

Before we present our main results, we first call out the typical incoherence assumption on the underlying data. 
\begin{definition}
A subspace $R(U)$ is incoherent with parameter $\mu$ if 
  \begin{equation}
    \max_{i \in \{1,\dots, n\}}\|\mathcal{P}_{U} e_i \|_2^2 \leq \frac{\mu d}{n} \nn
  \end{equation}
  where $e_i$ is the $i^{th}$ canonical basis vector and $\mathcal{P}_{U}$ is the projection operator onto the column space of $U$. 
\end{definition}

Note that $1 \leq \mu \leq \frac{n}{d}$. According to the above definition, the incoherence parameter of a vector $z\in \R^{n}$ is defined as: 
\begin{equation}
  \mu(z) = \frac{n \|z\|_{\infty}^2}{\|z\|_2^2} 
  \label{defn:vec_incoh}
\end{equation}
In this section, we assume the true subspace $R(\bar U)$ is incoherent with parameter $\mu_0$, and use $\mu(U)$, $\mu(v_\perp)$ to denote the incoherence parameter of $R(U)$ and $v_\perp$ respectively. We now show the expected improvement of $\zeta_t$ in a local region of the true subspace.

\begin{theorem}
Suppose $\sum_{k = 1}^{d}\sin^2\phi_k \leq \frac{d\mu_0}{16 n}$ and $\lvert \Omega\lvert = m$. If 
  \begin{equation}
    m > \max\left\{\frac{128d\mu_0}{3}\log\left(\sqrt{2d} n\right), 64 \mu(v_\perp)^2 \log\left(n\right), 52\left(1 + 2\sqrt{\mu(v_\perp)\log(n)}\right)^2d\mu_0\right\} \nn 
  \end{equation}
 then with probability at least $1 - \frac{3}{n^2}$ we have
  \begin{equation}
    \mathbb{E}_{v}\left[\zeta_{t + 1}\big\lvert U\right] \geq \left(1 + \frac{1}{4}\frac{m}{n}\frac{1 - \zeta_t}{d}\right) \zeta_t \;. \nn 
  \end{equation}
  \label{thm:miss_convgrate}
\end{theorem}
\vspace{-0.2cm}
This theorem shows that, within the local region of the true subspace, expected improvement on $\zeta_t$ can be obtained with high probability. As is implied by the theorem, this local region gets enlarged if the true subspace is more coherent, which may seem at first counterintuitive. However, the required number of measurements also increases as we increase $\mu_0$. In the extreme case, when $m$ increases to $n$, the local convergence results can be extended to a global result, as we proved for the full data case in Section \ref{sec:global_full}. On the other hand, compared to Theorem \ref{thm:det_convg_csrate}, the convergence result for the missing data case holds within a more conservative local region of the true subspace.
This gap is induced by the challenge of maintaining the incoherence property of our estimates $R(U)$, for which we had to consider the worst case.  We leave the extension of the local convergence results to global results as future work. 

In order to compare our result to the local convergence result in [Corollary 2.15, \cite{balzano2014local}], consider the following corollary. 
\begin{corollary}
  Define the \textit{determinant discrepancy} as $\kappa_t = 1 - \zeta_t$, then under the same conditions as Theorem \ref{thm:miss_convgrate}, we have 
  \begin{align}
  \mathbb{E}_v\left[\kappa_{t + 1} \big\lvert \kappa_t\right] \leq \left(1 -\frac{1}{4}\left(1 - \frac{d\mu_0}{16 n}\right)\frac{m}{nd}\right)\kappa_t \nn 
\end{align}
with probability exceeding $1 - 3/n^2$.
\label{coro:discrep_decay}
\end{corollary}

Recall that $1 \leq \mu_0 \leq \frac{n}{d}$, therefore the expected linear decay rate of $\kappa_t$ is at least $1 - \frac{9}{16}\frac{m}{n d}$. In \cite{balzano2014local} (Corollary 2.15), a similar linear convergence result is established in terms of the Frobenius norm discrepancy between $R(\bar U)$ and $R(U)$, denoted as $\epsilon_t = \sum_{i = 1}^{d}\sin^2\phi_{d}$. However, their result only holds when $\epsilon_t \leq (8 \times 10^{-6})\frac{m}{n^3d^2}$ which is more conservative than our assumption in Theorem \ref{thm:miss_convgrate}. Moreover, as we mentioned previously, empirical evidence shows the lower bound in Theorem \ref{thm:miss_convgrate} holds for every iteration from any random initialization. In contrast, in \cite{balzano2014local}, even for numerical results expected linear improvements only hold within the local region of the true subspace.  

Now we present the following intermediate results for the proof of Theorem \ref{thm:miss_convgrate}. 
{Note that in this missing data case, the projection residual $r_\Omega$ of $v_\Omega$ onto $U_\Omega$ is mapped back to $\R^n$ by zero padding the entries at the indices that are not in $\Omega$. Therefore, unlike Lemma \ref{lem:exp_delta_cs} of the compressively sampled data case, here $\|\widetilde r\| = \|r\| = \|r_{\Omega}\|$. Therefore, \eqref{eq:det_ratio_undersample} becomes
\begin{equation}
  \frac{\zeta_{t + 1}}{\zeta_t} \geq 1 + \frac{\left\|r_{\Omega}\right\|^2}{\|p\|^2} + 2\frac{\Delta}{\|p\|^2} \;.
  \label{eq:miss_det}
\end{equation}}
Now similarly to the compressively sampled data case, we proceed by establishing concentration results for the key quantities $\|r\|_2^2$, $\|p\|_2^2$ and $\Delta$ respectively.
\begin{lemma}[\cite{balzano2010high}, Theorem 1] Let $\delta>0$, and suppose $m \geq \frac{8}{3}d\mu(U)\log\left(2d/\delta\right)$. Then, with probability exceeding $1-3\delta$, 
  \begin{align}
  \left\|r_\Omega\right\|^2 &\geq (1 - \alpha_0)\frac{m}{n}\left\|v_\perp\right\|^2 \nn
\end{align} 
where $\alpha_0 = \sqrt{\frac{2\mu(v_\perp)^2}{m}\log\left(\frac{1}{\delta}\right)} + \frac{(\beta_1+1)^2}{1-\gamma_1}\frac{d\mu(U)}{m}$, $\beta_1 = \sqrt{2\mu(v_\perp)\log\left(\frac{1}{\delta}\right)}$, and $\gamma_1 = \sqrt{\frac{8d\mu(U)}{3m}\log\left(2d/\delta\right)}$.
  \label{lem:rconc_miss}
\end{lemma}

\begin{lemma}
Let $\delta>0$. Under the same condition on $m$ as Lemma \ref{lem:rconc_miss}, with probability at least $1 - 2\delta$ we have 
\begin{equation}
  \|p\|^2 \leq \left(1 + \frac{\beta_1 + 1}{1 - \gamma_1}\sqrt{\frac{d\mu(U)}{m}}\right)^2\|v\|^2 \nn 
\end{equation}
where $\beta_1$ and $\gamma_1$ equal to those defined in Lemma \ref{lem:rconc_miss}.
  \label{lem:pconc_miss}
\end{lemma}

\begin{lemma}
Let $\delta>0$. Under the same condition on $m$ as Lemma \ref{lem:rconc_miss}, with probability at least $1 - 3\delta$ we have
\begin{equation}
  \left\lvert\Delta\right\lvert \leq  \frac{\eta_3}{\cos\phi_{d}}\sqrt{\sin^2\phi_d + \frac{  d\mu_0}{m}}\sqrt{\frac{d\mu(U)}{m}}\frac{m}{n}\|v_{\perp}\|^2  \nn
\end{equation}
where $\eta_3 = \frac{(1 + \beta_1)(1 + \beta_2)}{1 - \gamma_1}$, $\beta_2 = \sqrt{2\mu(v_\perp)\log\left(\frac{1}{\delta}\right)\frac{d\mu_0}{d\mu_0 + m\sin^2\phi_d}}$, and $\beta_1$ and $\gamma_1$ equal to those defined in Lemma \ref{lem:rconc_miss}.
  \label{lem:delta_miss}
\end{lemma}
Lemma \ref{lem:rconc_miss} shows that the concentration of $\|r\|_2^2 = \|r_\Omega\|_2^2$ does not only depend on the sampling framework, but also on the incoherence property of the current estimate and the true projection residual, \ie $\mu(U)$ and $\mu(v_\perp)$. To see this clearly, recall that $\|r_\Omega\|_2^2 = \left\|v_{\perp, \Omega}\right\|_2^2 - \left\|\mathcal{P}_{U_\Omega}\left(v_{\perp, \Omega}\right)\right\|_2^2$, hence the incoherence property of $v_\perp$ and $R(U)$ directly influences the concentration of $\|r_\Omega\|_2^2$. On the other hand, for compressive data, the Gaussian distributed sampling matrices yield tight concentration results for $\|p\|_2^2$, $\|r_\Omega\|_2^2$ and $\Delta$. Therefore, the upper bounds of the key quantities established in Lemmas \ref{lem:rconc_miss}, \ref{lem:pconc_miss} and \ref{lem:delta_miss} are not as tight as those for the compressive data except the extreme case where $\mu(U) = \mu(v_\perp) = 1$, \ie both $R(U)$ and $v_\perp$ are incoherent.

As shown in the above lemmas, in order to establish concentration of the key quantities in (\ref{eq:miss_det}), it is essential for the subspaces generated by GROUSE to be incoherent over iterates. It has been proven in \cite{balzano2014local} that within the local region of $R(\bar U)$, the incoherence of $R(U)$ can be bounded by that of $R(\bar U)$.

\begin{lemma}[\cite{balzano2014local}, Lemma 2.5]
Suppose $\sum_{k = 1}^{d}\sin^2\phi_k \leq \frac{d}{16 n}\mu_0$, then $\mu(U) \leq 2\mu_0$. 
  \label{lem:uincoh}
\end{lemma}
Now we are ready to prove Theorem \ref{thm:miss_convgrate}. We sketch the proof here, and a detailed proof is provided in Section \ref{sec:proof_of_undersampled_data}.
\begin{proof}[{Proof sketch of Theorem \ref{thm:miss_convgrate}}]
Given the condition required by Theorem \ref{thm:miss_convgrate}, we have $\sin\phi_d \leq \sqrt{d\mu_0/16n}$ and $\cos\phi_d \geq \sqrt{1 - d\mu_0 / 16n}$. This together with Lemma \ref{lem:uincoh} and Lemma \ref{lem:delta_miss} yield $\left\lvert\Delta\right\lvert \leq \frac{11}{5}\eta_3\frac{d\mu_0}{n}\left\|v_\perp\right\|^2$. Also for $\beta_2$ in Lemma \ref{lem:delta_miss}, $\beta_2 \leq \sqrt{2\mu(v_\perp)\log(1/\delta)} = \beta_1$. Hence, 
\begin{equation}
  \left\lvert\Delta\right\lvert \leq \frac{11}{5}\frac{(1 + \beta_1)^2}{1-\gamma_1}\frac{d\mu_0}{n}\|v_\perp\|^2 \;.
  \label{eq:simp_delta_miss_sketch}
\end{equation} 
Letting $\eta_2 = \frac{(1 + \beta_1)^2}{1-\gamma_1}\frac{d\mu_0}{m}$ and $\alpha_1 = \sqrt{\frac{2\mu(v_\perp)^2}{m}\log\left(\frac{1}{\delta}\right)}$, then applying this definition together with Lemma \ref{lem:uincoh} to Lemma \ref{lem:pconc_miss} and Lemma \ref{lem:rconc_miss} yields
\begin{align}
  &\left\|p\right\|^2 \leq \left(1 + \sqrt{\frac{2\eta_2}{1 - \gamma_1}}\right)^2\|v\|^2 \label{eq:pmiss_simp_sketch} \\
  & \left\|r_\Omega\right\|^2 \geq (1 - \alpha_1 - 2\eta_2)\frac{m}{n}\left\|v_\perp\right\|^2 
  \label{eq:rmiss_simp_sketch}
\end{align}
Now applying \eqref{eq:simp_delta_miss_sketch}, \eqref{eq:pmiss_simp_sketch} and \eqref{eq:rmiss_simp_sketch} to \eqref{eq:miss_det} we have
\begin{align}
  \frac{\zeta_{t+1}}{\zeta_t} 
  &\geq 1 + \frac{(1 - \alpha_1 - \frac{32}{5}\eta_2)}{(1 + \sqrt{2\eta_2/(1 -\gamma_1)})^2}\frac{m}{n} \frac{\|v_\perp\|^2}{\|v\|^2}
  \label{eq:miss_pf_det_raw_sketch}
\end{align}
with probability at least $1 - 3\delta$. The probability bound is obtained by taking the union bound of those generating Lemmas \ref{lem:rconc_miss}, \ref{lem:pconc_miss} and \ref{lem:delta_miss}, as we can see in the proofs in Section \ref{sec:proof_of_undersampled_data} this union bound is at least $1 - 3\delta$. 

Letting $\eta_1 = \frac{(1 - \alpha_1 - \frac{32}{5}\eta_2)}{(1 + \sqrt{2\eta_2/(1-\gamma_1)})^2}$, then $\eta_1> 0$ is equivalent to $1 - \alpha_1 - \frac{32}{5}\eta_2 >0$. This further gives that if $m$ satisfies the condition in Theorem \ref{thm:miss_convgrate}, then $\eta_1 > \frac{1}{4}$. Now taking expectation with respect to $v$ yields, 
\begin{equation}
  \mathbb{E}_v\left[\zeta_{t + 1} \big\lvert U\right] \geq \left(1 + \frac{1}{4}\frac{m}{n}\mathbb{E}\left[\frac{\|v_\perp\|^2}{\|v\|^2}\big\lvert U\right]\right)\zeta_t \geq \left(1 + \frac{1}{4}\frac{m}{n}\frac{1 - \zeta_t}{d}\right)\zeta_t
\end{equation}
where the last inequality follows from Lemma \ref{lem:key_quantity}. Finally choosing $\delta$ to be $1/n^2$completes the proof. 
\end{proof}

\section{Numerical Results}
\label{sec:experment}
\setlength\figurewidth{0.85\columnwidth} 
\setlength\figureheight{0.35 \columnwidth} 
\begin{figure}[!ht]\footnotesize
  \begin{center}
    \ifbool{DrawFigures}{
%
\begin{tikzpicture}

\begin{axis}[%
width=0.444444444444444\figurewidth,
height=0.555555555555556\figureheight,
at={(0.555555555555556\figurewidth,0\figureheight)},
scale only axis,
view={0}{90},
every outer x axis line/.append style={white!15!black},
every x tick label/.append style={font=\color{white!15!black}},
xmin=10,
xmax=100,
xtick={ 10,  40,  70, 100},
xlabel={$d$},
xmajorgrids,
every outer y axis line/.append style={white!15!black},
every y tick label/.append style={font=\color{white!15!black}},
ymin=500,
ymax=5000,
ytick={ 500, 2500, 4500},
ylabel={$n$},
ymajorgrids,
every outer z axis line/.append style={white!15!black},
every z tick label/.append style={font=\color{white!15!black}},
zmin=3.29698575782831e-06,
zmax=0.000879890274018717,
zmajorgrids,
title style={font=\bfseries},
title={$\widehat{\mathrm{Var}}\left[K/(d^2\log(n) + d\log(1 - \zeta^\ast))\right]$},
axis x line*=bottom,
axis y line*=left,
axis z line*=left,
scaled ticks = false, xmajorgrids=false, ymajorgrids=false, axis x line=bottom, axis y line=left, every axis x label/.style={at={(current axis.south east)},anchor=west},  every axis y label/.style={at={(current axis.north west)},anchor=south},
colormap/jet,
colorbar horizontal,
colorbar style={separate axis lines,every outer x axis line/.append style={white!15!black},every x tick label/.append style={font=\color{white!15!black}},every outer y axis line/.append style={white!15!black},every y tick label/.append style={font=\color{white!15!black}}},
point meta min=3.29698575782831e-06,
point meta max=0.000879890274018717
]

\addplot3[%
surf,
faceted color=black,
shader=faceted,
colormap/jet,
mesh/rows=10]
table[row sep=crcr,header=false] {%
10	500	0.000596323240543501\\
10	1000	0.000271494386419479\\
10	1500	6.71855788220541e-05\\
10	2000	7.60221813327341e-05\\
10	2500	3.69694583016502e-05\\
10	3000	1.54474301094484e-05\\
10	3500	1.21255121872094e-05\\
10	4000	6.23807803622616e-06\\
10	4500	7.40480171693545e-06\\
10	5000	5.57549541568989e-06\\
20	500	0.000670113574539312\\
20	1000	0.000123711063901791\\
20	1500	0.000106765063132829\\
20	2000	5.69622157374721e-05\\
20	2500	3.06386607737317e-05\\
20	3000	1.36780785384464e-05\\
20	3500	1.28576512218354e-05\\
20	4000	6.69382252336445e-06\\
20	4500	4.81277979501163e-06\\
20	5000	6.50923878050508e-06\\
30	500	0.000781578626285758\\
30	1000	0.000207121115317456\\
30	1500	9.04068600129132e-05\\
30	2000	2.122602373229e-05\\
30	2500	3.99239434162461e-05\\
30	3000	2.02930346635521e-05\\
30	3500	2.34468119838972e-05\\
30	4000	1.32007610008885e-05\\
30	4500	7.63586561596731e-06\\
30	5000	4.59788315712924e-06\\
40	500	0.000715832927650073\\
40	1000	8.82650417277107e-05\\
40	1500	5.20470166830703e-05\\
40	2000	7.32853863398004e-05\\
40	2500	1.46293179187546e-05\\
40	3000	1.88858963336718e-05\\
40	3500	7.89997659592903e-06\\
40	4000	1.27969045581798e-05\\
40	4500	6.76573281282173e-06\\
40	5000	9.76538997808435e-06\\
50	500	0.000354235526177687\\
50	1000	0.000209237532128569\\
50	1500	7.31043686422187e-05\\
50	2000	7.04539779021705e-05\\
50	2500	3.14887875678796e-05\\
50	3000	1.31910000454795e-05\\
50	3500	9.07466092549343e-06\\
50	4000	1.89665493749507e-05\\
50	4500	5.06095687681189e-06\\
50	5000	3.73813969795992e-06\\
60	500	0.000746513427947547\\
60	1000	0.000354231098486113\\
60	1500	6.24304573920003e-05\\
60	2000	3.2152425675106e-05\\
60	2500	2.12261532890704e-05\\
60	3000	7.80701052865077e-06\\
60	3500	1.20772113252092e-05\\
60	4000	9.39462515616341e-06\\
60	4500	4.94921977827458e-06\\
60	5000	7.33756986143377e-06\\
70	500	0.000663047300947332\\
70	1000	0.000117183860362591\\
70	1500	4.88924587911712e-05\\
70	2000	3.21908981119418e-05\\
70	2500	2.64388625548206e-05\\
70	3000	9.32159927521496e-06\\
70	3500	1.03191075287021e-05\\
70	4000	7.50784110311827e-06\\
70	4500	4.94067229452248e-06\\
70	5000	4.17467291604087e-06\\
80	500	0.000429670584596072\\
80	1000	0.000153996947280239\\
80	1500	8.12407431144762e-05\\
80	2000	2.89409713494193e-05\\
80	2500	3.24511520617512e-05\\
80	3000	1.78620679186788e-05\\
80	3500	1.62610955133431e-05\\
80	4000	6.43129298694604e-06\\
80	4500	4.21642385411436e-06\\
80	5000	6.61617062411983e-06\\
90	500	0.000879890274018717\\
90	1000	7.64745950242355e-05\\
90	1500	7.47456871566045e-05\\
90	2000	3.16978703650528e-05\\
90	2500	2.8415960388248e-05\\
90	3000	1.06458628334069e-05\\
90	3500	1.24806298748132e-05\\
90	4000	8.78892153205104e-06\\
90	4500	3.52255499927056e-06\\
90	5000	5.49766601787916e-06\\
100	500	0.000445197158885553\\
100	1000	0.000101072810714596\\
100	1500	0.000114906258476765\\
100	2000	1.28151787693256e-05\\
100	2500	1.35863517725944e-05\\
100	3000	1.25702607577597e-05\\
100	3500	1.94229421243486e-05\\
100	4000	7.41177534821017e-06\\
100	4500	3.29698575782831e-06\\
100	5000	5.95643273765954e-06\\
};
\end{axis}

\begin{axis}[%
width=0.444444444444445\figurewidth,
height=0.555555555555556\figureheight,
at={(0\figurewidth,0\figureheight)},
scale only axis,
view={0}{90},
every outer x axis line/.append style={white!15!black},
every x tick label/.append style={font=\color{white!15!black}},
xmin=10,
xmax=100,
xtick={ 10,  40,  70, 100},
xlabel={$d$},
xmajorgrids,
every outer y axis line/.append style={white!15!black},
every y tick label/.append style={font=\color{white!15!black}},
ymin=500,
ymax=5000,
ytick={ 500, 2500, 4500},
ylabel={$n$},
ymajorgrids,
every outer z axis line/.append style={white!15!black},
every z tick label/.append style={font=\color{white!15!black}},
zmin=0.0272012156562988,
zmax=0.23819393326667,
zmajorgrids,
title style={font=\bfseries},
title={$\widehat{\mathrm{E}}\left[K/(d^2\log(n) + d\log(1 - \zeta^\ast))\right]$},
axis x line*=bottom,
axis y line*=left,
axis z line*=left,
scaled ticks = false, xmajorgrids=false, ymajorgrids=false, axis x line=bottom, axis y line=left, every axis x label/.style={at={(current axis.south east)},anchor=west},  every axis y label/.style={at={(current axis.north west)},anchor=south},
colormap/jet,
colorbar horizontal,
colorbar style={separate axis lines,every outer x axis line/.append style={white!15!black},every x tick label/.append style={font=\color{white!15!black}},every outer y axis line/.append style={white!15!black},every y tick label/.append style={font=\color{white!15!black}}},
point meta min=0.0272012156562988,
point meta max=0.23819393326667
]

\addplot3[%
surf,
faceted color=black,
shader=faceted,
colormap/jet,
mesh/rows=10]
table[row sep=crcr,header=false] {%
10	500	0.23819393326667\\
10	1000	0.223235879276348\\
10	1500	0.216170148295355\\
10	2000	0.218338480623332\\
10	2500	0.210861420349189\\
10	3000	0.209243415126034\\
10	3500	0.204186825688651\\
10	4000	0.205098517407477\\
10	4500	0.208760436986486\\
10	5000	0.207524847540074\\
20	500	0.144167085479711\\
20	1000	0.13233540970703\\
20	1500	0.131135979190733\\
20	2000	0.125794636550435\\
20	2500	0.125253035431963\\
20	3000	0.127260504028832\\
20	3500	0.117454957795827\\
20	4000	0.119927885365199\\
20	4500	0.120054630775443\\
20	5000	0.11971315344987\\
30	500	0.102354717909339\\
30	1000	0.093347646086628\\
30	1500	0.0906456345370041\\
30	2000	0.0896148698780685\\
30	2500	0.0871099864671774\\
30	3000	0.0873914837863311\\
30	3500	0.0867234445713873\\
30	4000	0.0850239636134074\\
30	4500	0.0836587287958704\\
30	5000	0.0848927637564892\\
40	500	0.0792555839797657\\
40	1000	0.0749770896882569\\
40	1500	0.0687403880608508\\
40	2000	0.0688515409331148\\
40	2500	0.0699082706136762\\
40	3000	0.066630741139975\\
40	3500	0.065642512324239\\
40	4000	0.064780428353251\\
40	4500	0.0643579763381401\\
40	5000	0.0644068527165736\\
50	500	0.0643056575402827\\
50	1000	0.0602071276557131\\
50	1500	0.0591136117098092\\
50	2000	0.0557765667475368\\
50	2500	0.0558157536741937\\
50	3000	0.0539677251464824\\
50	3500	0.0544184096753541\\
50	4000	0.0534134979444562\\
50	4500	0.0545119258214879\\
50	5000	0.0511672451243724\\
60	500	0.0532586915982165\\
60	1000	0.0507384984382984\\
60	1500	0.0499450392221963\\
60	2000	0.0490270720130192\\
60	2500	0.0465960280747145\\
60	3000	0.045165688906335\\
60	3500	0.0461012670530287\\
60	4000	0.045397749501713\\
60	4500	0.044624299317212\\
60	5000	0.0448028732519438\\
70	500	0.0469101844514621\\
70	1000	0.0452839455440687\\
70	1500	0.0449211154968827\\
70	2000	0.0407082282871925\\
70	2500	0.0414724031622326\\
70	3000	0.0413424057370412\\
70	3500	0.0398043248249403\\
70	4000	0.039097601934874\\
70	4500	0.0388874387037527\\
70	5000	0.0404870618897745\\
80	500	0.040615166000303\\
80	1000	0.0381416755537641\\
80	1500	0.0379703844420277\\
80	2000	0.0374207356875253\\
80	2500	0.0366451529204431\\
80	3000	0.0356378015037752\\
80	3500	0.0351180980008267\\
80	4000	0.0344978935475362\\
80	4500	0.0348184776426137\\
80	5000	0.0342493046933627\\
90	500	0.037346149486389\\
90	1000	0.0343625675834102\\
90	1500	0.0344853903406754\\
90	2000	0.0336072597962442\\
90	2500	0.0317510241290618\\
90	3000	0.0322393974512222\\
90	3500	0.0317031868811921\\
90	4000	0.0305339061988143\\
90	4500	0.0301888414396341\\
90	5000	0.0299683101064155\\
100	500	0.0335335882408509\\
100	1000	0.031716830781452\\
100	1500	0.0301251738571533\\
100	2000	0.0307413431806563\\
100	2500	0.0293539179000507\\
100	3000	0.0288037911713605\\
100	3500	0.0282811139535851\\
100	4000	0.0283936307969904\\
100	4500	0.0287047469609287\\
100	5000	0.0272012156562988\\
};
\end{axis}
\end{tikzpicture}%
    }{
      \textcolor{green}{\rule{\figurewidth}{\figureheight}}
    }
  \end{center}
  \caption{Illustration of the bounds on $K$ in Conjecture \ref{thm:global} compared to their values in practice, averaged over $50$ trials with different $n$ and $d$. We show the ratio of $K$ to the bound $d^2\log(n) + d\log(1 - \zeta^\ast)$.}
\label{fig:Knfree}
\end{figure}
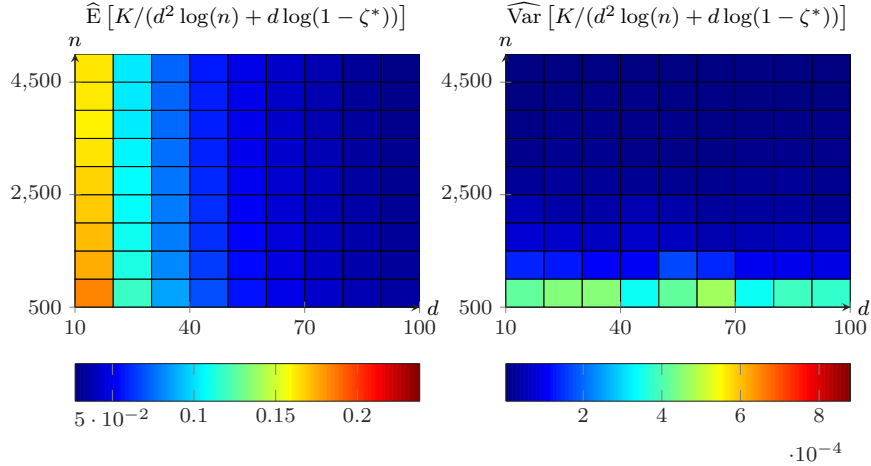
\sloppypar In this section, we demonstrate that our theoretical results match the empirical convergence behavior of GROUSE. We generate the underlying data matrix $M = \left[\begin{matrix} v_1 & v_2 & \dots & v_T \end{matrix}\right]$ as $M = \bar U W$. For both the fully sampled data case and compressively sampled data case, the underlying signals are generated from a sparse subspace, demonstrating that incoherence assumptions are not required by our results for these two cases. Specifically, the underlying subspace of each trial is set to be a sparse subspace, as the range of an $n\times d$ matrix $\bar U$ with sparsity on the order of $\frac{\log(n)}{n}$. For the missing data case, we generate the underlying subspace as the range of an $n\times d$ matrix with i.i.d standard normal distribution.  The entries of the coefficient matrix $W$ for all three cases are generated as i.i.d $\mathcal{N}(0,1)$ satisfying Assumption \ref{cond:v}. We also want to mention that we run GROUSE with random initialization for all of the plots in this section. 

\begin{figure}[!ht]\footnotesize
\setlength\figurewidth{0.92\columnwidth} 
\setlength\figureheight{0.35 \columnwidth}
  \input{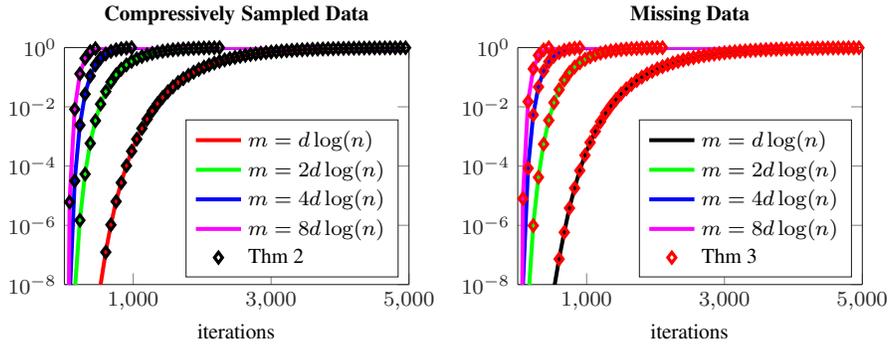}
\caption{Illustration of expected improvement on $\zeta$ given by Theorem \ref{thm:det_convg_csrate} \text{(left)} and Theorem \ref{thm:miss_convgrate} \text{(right)} over $50$ trials. We set $n = 5000$, $d = 10$. The diamonds denote the lower bound on expected convergence rates described in Theorem \ref{thm:det_convg_csrate} and Theorem \ref{thm:miss_convgrate}.}
\label{fig:det_ratio}
\end{figure} 

\setlength\figurewidth{0.90\columnwidth} 
\setlength\figureheight{0.90 \columnwidth} 
\begin{figure}[!ht]\footnotesize
  \begin{center}
    \ifbool{DrawFigures}{
      \input{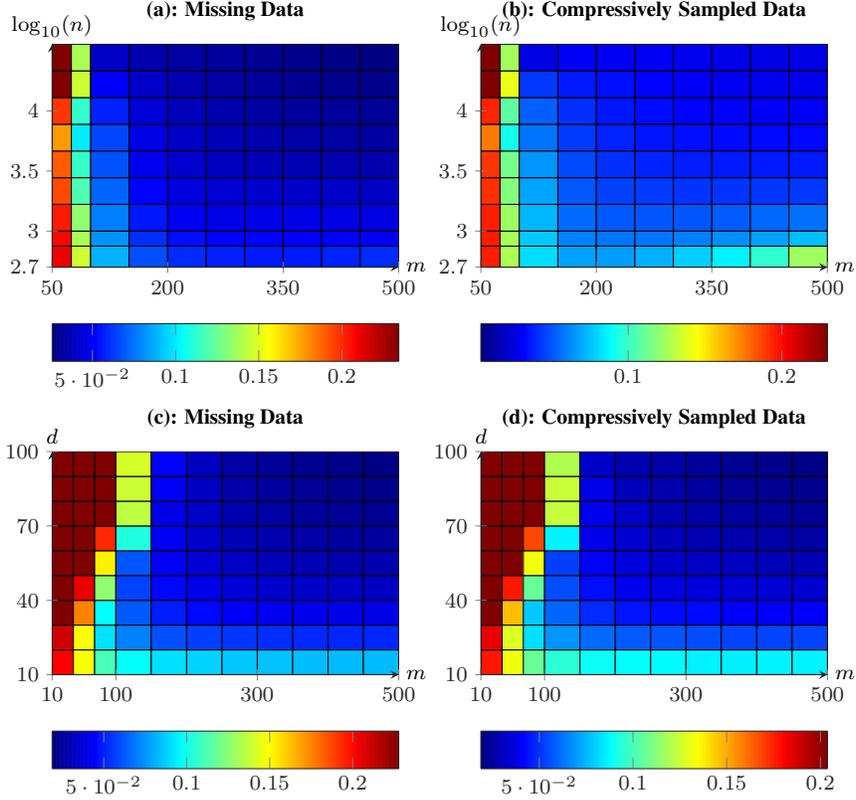}
    }{
      \textcolor{green}{\rule{\figurewidth}{\figureheight}}
    }
  \end{center}
  \caption{Illustration of our heuristic bounds on $K$ (the actual iterations required by GROUSE to converge to the given accuracy) over different $d$, $m$ and $n$, averaged over 20 trials. In this simulation, we run GROUSE from a random initialization to convergence for a required accuracy $\zeta^\ast = 1 - 1e$-3. We show the ratio of $K$ to the heuristic bound $\frac{n}{m}\left(d^2\log(n) + d\log(1 - \zeta^\ast)\right)$.   In {(a)} and {(b)}, we set $d = 50$ and examine $K$ over $m$ and $n$ for both \text{missing data {(a)}} and \text{compressively sampled data {(b)}}. In {(c)} and {(d)}, we set $n = 10000$ and examine $K$ over $m$ and $d$ for both \text{missing data {(c)}} and \text{compressively sampled data {(d)}}. In these plots, we use the dark red to indicate the failure of convergence.}
  \label{fig:heuristic}
\end{figure}

We first examine our global convergence result, \ie Theorem \ref{thm:asymp_full} and Conjecture \ref{thm:global}, for the fully sampled data in Figure \ref{fig:Knfree}. We run GROUSE to convergence for a required accuracy $\zeta^\ast = 1 - 1e$-4 and show the ratio of $K$ to the simplified bound of Conjecture \ref{thm:global}, $d^2\log(n) + d\log\frac{1}{1 - \zeta^\ast}$. We run GROUSE over $50$ trials and show the mean and variance. We can see that, for fixed $n$, despite the conjecture's tighter convergence rate than the theorem's, it becomes loose as we increase the dimension of the underlying subspace. However, compared to the empirical mean, the empirical variance is very small. This indicates that the relationship between our conjectured upper bounds and the actual iterations required by GROUSE is stable.  

Next we examine our theoretical results (Theorem \ref{thm:det_convg_csrate} and Theorem \ref{thm:miss_convgrate}) for the expected improvement on $\zeta_t$ for the undersampled case in Figure \ref{fig:det_ratio}. We set $n = 5000$ and $d = 10$. We run GROUSE over different sampling numbers $m$. The plots are obtained by averaging over $50$ trials. We can see that our theoretical bounds on the expected improvement on $\zeta_t$ for both missing data and compressively sampled data are tight from any random initialization, although we have only established local convergence results for both cases. Also note that Theorem \ref{thm:det_convg_csrate} and Theorem \ref{thm:miss_convgrate} indicate that the expected improvement on the determinant similarity has a similar form to that of the fully sampled case roughly scaled by the sampling density $(m/n)$. These together motivate us to approximate the required iterations to achieve a given accuracy as that required by the fully sampled case times the reciprocal of sampling density, $n/m$: 
\vspace{-0.0cm}
\begin{equation}
	 \left(n/m\right) \cdot \left( d^2\log(n) + d\log(1 - \zeta^\ast)\right)\;. \nn
\end{equation}
As we see in Figure \ref{fig:heuristic}, when $m$ is slightly larger than $d$, the empirical mean of the ratio of the actual iterations required by GROUSE to our heuristic bound is similar to that of the full data case. We leave the rigorous proof of this heuristic as future work. 

\section{Conclusion}
\label{sec:conclustion}
\revise{In this paper, we analyze a manifold incremental gradient descent algorithm applied to a particular non-convex optimization formulation for recovering a low-dimensional subspace from streaming data sampled from that subspace. We provide a simplified analysis as compared to \cite{zhang2015global}, showing global convergence of the algorithm to the global minimizer for fully sampled data. However, the convergence rate we have established in theory is loose compared to what we observed in practice. A future direction is to narrow the gap between our theory and the actual performance of GROUSE, for which Conjecture \ref{thm:global} shows great promise. 

With undersampled data, we show that expected improvement on our defined convergence metric can be obtained with high probability for each iteration. We prove that, comparing with fully sampled data, the expected improvement on determinant similarity is roughly proportional to the sampling density. With compressively sampled data this expected improvement holds from any random initialization, while it only holds within the local region of the true subspace for the missing data case. The limitation on the convergence of missing data arises due to the challenge of maintaining the incoherence property of our estimates in theory. Crossing this fundamental hurdle and extending the local convergence with missing data to a global result would be an interesting and valuable future direction.
} 

\appendix
\section{Supplementary material}
\subsection{Preliminaries}
\label{sec:appendix_a}
We start by providing the following lemma that we will use regularly in the manipulation of the matrix $\bar{U}^T U$. It also provides us with more insight into our metric of determinant similarity between the subspaces. The proof can be found in \cite{stewart1990matrix}. 
\begin{lemma}[\cite{stewart1990matrix}, Theorem 5.2]
Let $U, \bar U \in \R^{n\times d}$ with orthonormal columns, then  there are unitary matrices $Q$, $\bar{Y}$, and $Y$ such that 
\begin{equation}
  Q\bar{U}\bar{Y} := \bordermatrix{&d\cr
                d& I \cr
               d & 0 \cr
                n-2d & 0 }~~\text{and}~~
  Q U Y := \bordermatrix{&d\cr
                d& \Gamma \cr
               d & \Sigma \cr
                n-2d & 0 }   \nonumber
\end{equation}
where $\Gamma = \diag{(\cos\phi_{1}, \dots, \cos\phi_{d})}, \Sigma = \diag{(\sin\phi_{1}, \dots, \sin\phi_{d})}$ with $\phi_{ i}$ being the $i^{th}$ principal angle between $R(U)$ and $R(\bar U)$ defined in Definition \ref{defn:detdiscrepancy}.
\label{lem:subsprelate}
\end{lemma}
Now we are going to prove Lemma \ref{lem:key_quantity}, which is essential for us to establish expected improvement on the determinant similarity for each iteration in the various sampling cases we consider. Before that, we present the following lemmas that are requried for the proof.

\revise{\begin{lemma} Given any matrix $Q \in \R^{d \times d}$, suppose that $w \in \R^d$ is a random vector whose components $w_i$, $i=1,\dots,d$ are zero-mean, independent, and identically distributed symmetrically about zero (\ie the distribution of $w_i$ is an even function). Then $$E\left[ \frac{w^T Q w}{w^Tw} \right] = \frac{1}{d} \trace (Q)\;.$$
\label{lem:expisovec}
\end{lemma}

\begin{proof}[{Proof of Lemma \ref{lem:expisovec}}]
\begin{align}
E\left[ \frac{w^T Q w}{w^Tw} \right] &= \sum_{i\neq j} E\left[ \frac{w_i w_j Q_{ij}}{w^Tw} \right] + \sum_{i=1}^d E\left[ \frac{w_i^2 Q_{ii}}{w^Tw} \right] \nn \\
&= \sum_{i=1}^d  Q_{ii} E\left[ \frac{w_i^2}{w^Tw} \right] \label{step:symmarg} \\
&= \frac{1}{d} \trace Q\label{step:trick}\;,
\end{align}
where Eqs (\ref{step:symmarg}) and (\ref{step:trick}) hold by the following two arguments. For Eq (\ref{step:symmarg}), let $f(w_1, \dots, w_d)$ be the joint distribution among the coordinates, and without loss of generality let $i=1$ and $j \neq 1$, then
\begin{align}
&E\left[ \frac{w_1 w_j Q_{1j}}{w^Tw} \right]\nn \\
& = \int_{-\infty}^\infty \cdots \int_{-\infty}^\infty  \frac{w_1 w_j Q_{ij}}{w^Tw} f(w_1, \dots, w_d) d w_1 d w_2 \cdots dw_d \nn\\
&= \int_{-\infty}^\infty \cdots \int_{-\infty}^\infty  \frac{w_1 w_j Q_{1j}}{w_1^2 + \sum_{k\neq i} w_k^2} f(w_1) f(w_2) \cdots f(w_d) d w_1 d w_2 \cdots dw_d \nn\\
&= \int_{-\infty}^\infty \cdots \int_{-\infty}^\infty \left(  \int_{-\infty}^\infty  \frac{w_1 }{w_1^2 + \sum_{k\neq i} w_k^2} f(w_1) dw_1 \right) w_j Q_{1j} f(w_2) \cdots f(w_d) d w_2 \cdots dw_d \nn \\
&= 0 \nn
\end{align} 
where the last inequality holds since $\frac{w_1 }{w_1^2 + \sum_{k\neq i} w_k^2}$ is an odd function of $w_1$ and $f(w_1)$ is an even function of $w_1$, thereby the term in parentheses will integrate to zero. We note that if $w_i$ is a discrete random variable, the argument would be similar. 

To get Eq (\ref{step:trick}) we note that 
\begin{align}
1 = E \left[ \frac{\sum_i w_i^2 }{\sum_j w_j^2} \right] = \sum_i E \left[ \frac{w_i^2}{w^Tw} \right] = d E \left[ \frac{w_i^2}{w^Tw} \right] \;, i=1,\dots, d\;,
\nn
\end{align} 
where the last step holds because each $w_i$ is identically distributed.
\end{proof}}

\begin{lemma}[\cite{sa2015global}, Lemma 16]
Let $X = [X_1, \cdots, X_d]$ with $X_i \in [0,1], i = 1, \dots, d$, then 
\begin{align}
   d - \sum_{i = 1}^{d}X_i \geq  1 - \Pi_{i = 1}^{d}X_i \nn 
\end{align}
  \label{lem:zeta_eps_relate}
\end{lemma}

\begin{proof}[{Proof of Lemma \ref{lem:key_quantity}}]
  According to Lemma \ref{lem:expisovec} and Lemma \ref{lem:zeta_eps_relate} we have the following 
  \begin{align}
    \mathbb{E}\left[\frac{\|v_{\perp}\|^2}{\|v\|^2} \bigg\lvert U\right] = \mathbb{E}\left[\frac{\|\bar U s\|^2 - \|UU^T\bar U s\|^2}{\|\bar U s\|^2}\bigg\lvert U\right] &\overset{\vartheta_1}= \mathbb{E}\left[\frac{s^T\bar Y(I - \Gamma^2)\bar Y^T s}{s^Ts}\bigg\lvert U\right] \nn \\
    &\overset{\vartheta_2}{=} \frac{1}{d}\trace\left(I - \Gamma^2\right) \overset{\vartheta_3}\geq  \frac{1 - \zeta_t}{d} 
    \label{eq:vconc} 
  \end{align}
  where $\vartheta_1$ follows by Lemma \ref{lem:subsprelate} and $\|\bar U s\|^2 = \|s\|^2$, $\vartheta_2$ from Lemma \ref{lem:expisovec}, and $\vartheta_3$ from Lemma \ref{lem:zeta_eps_relate} with $X_i = \cos^2\phi_{i}$.
\end{proof}

\subsection{Proof of Fully Sampled Data}
\label{sec:proof_full_nfree}
In this section we prove the results of Section \ref{sec:global_full}. We start by proving Eq \ref{eq:det_ratio_expr}, the deterministic expression for the change in determinant similarity from one step of the GROUSE algorithm to the next. Using this expression, we prove the GROUSE monotonic improvement of Lemma \ref{lem:mono_full}, expected improvement of Lemma \ref{lem:exp_mono_zeta}, and finally the global convergence \new{conjecture} \ref{thm:global}. 

Recall that $\frac{y}{\|y\|} = \cos (\theta) \frac{v_\para}{\|v_\para\|} + \sin(\theta) \frac{v_\perp}{\|v_\perp\|}$ in Algorithm \ref{alg:grouse}. Then according to the GROUSE update in \ref{eq:gpupdate} we have 

\begin{align}
  \det \left(\bar U^T U_{t + 1}\right) &= \det\left(\bar U^T U + \left(\frac{\bar U^T y}{\|y\|} - \frac{\bar U^T v_{\parallel}}{\|v_{\parallel}\|}\right)\frac{w^T}{\|w\|}\right) \nn \\
  &\overset{\vartheta_1}= \det\left(\bar U^T U\right)\left(1 + \frac{w^T(\bar U^T U)^{-1}}{\|w\|}\left(\frac{\bar U^T y}{\|y\|} - \frac{\bar U^T v_{\parallel}}{\|v_{\parallel}\|}\right)\right) \nn \\
      &\overset{\vartheta_2}= \det\left(\bar U^T U\right)\frac{w^T(\bar U^TU)^{-1} \bar U^T y}{\|y\|\|w\|} \nn \\
      &\overset{\vartheta_3}= \det\left(\bar U^T U\right)\left(\cos\theta + \frac{\|v_{\perp}\|}{\|v_{\parallel}\|}\sin\theta\right) 
  \label{eq:full_det_ratio_prove}
\end{align}
where $\vartheta_1$ follows from the Schur complement, \ie that for any invertible matrix $M$ we have $\det\left(M + a b^T\right) = \det(M)\left(1 + b^TM^{-1}a\right)$; $\vartheta_2$ and $\vartheta_3$ hold since $\|v_{\para}\|^2 = \|U w\|^2 = \|w\|^2$ and the following
\begin{subequations}
  \begin{align}
  & w^T(\bar U^TU)^{-1}\bar U^Tv_{\parallel} \overset{w = U^T\bar U s}= v^Tv_{\parallel} = \|v_{\parallel}\|^2 \\
  & w^T(\bar U^T U)^{-1}\bar U^T v_{\perp} \overset{w = U^T\bar U s}= v^Tv_{\perp} = \|v_{\perp}\|^2.
\end{align} 
\label{eq:fullnfree_eqs}
\end{subequations}

Given this, the proof of Lemma \ref{lem:mono_full} follows directly from the above proof and the greedy step size derived in Eq. \ref{eq:stepsize_greedy}. 
\begin{proof}[Proof of Lemma \ref{lem:mono_full}]
By using $\theta = \arctan\left(\frac{\|v_{\perp}\|}{\|v_{\parallel}\|}\right)$, we have $\cos\theta=\frac{\|v_\para\|}{\|v\|}$ and $\sin\theta = \frac{\|v_\perp\|}{\|v\|}$. This together with \ref{eq:full_det_ratio_prove} gives $\det\left(\bar U^TU_{t + 1}\right)=\det\left(\bar U^TU\right)\frac{\|v\|}{\|v_\para\|}$. Therefore, $\frac{\zeta_{t + 1}}{\zeta_t} = \frac{\det\left(\bar U^TU_{t + 1}\right)^2}{\det\left(\bar U^TU\right)^2} = \frac{\|v\|^2}{\|v_{\parallel}\|^2} = 1 + \frac{\|v_{\perp}\|^2}{\|v_{\parallel}\|^2}$.
\end{proof}

\begin{proof}[Proof of Lemma \ref{lem:exp_mono_zeta}]
Lemma \ref{lem:exp_mono_zeta} follows directly from \ref{lem:key_quantity} and \ref{lem:mono_full}, \ie
  \begin{align}
    \mathbb{E}\left[\frac{\zeta_{t + 1}}{\zeta_t}\bigg\lvert U\right] = 1 + \mathbb{E}\left[\frac{\|v_{\perp}\|^2}{\|v_{\parallel}\|^2}\bigg\lvert U\right] &\geq 1 + \mathbb{E}\left[\frac{\|v_{\perp}\|^2}{\|v\|^2}\bigg\lvert U\right] \nn \\
    &\geq 1 + \frac{1 - \zeta_t}{d}
  \end{align}
  Note that, given $U$, $\zeta_t$ is a constant, hence completes the proof.
\end{proof}

\revise{
With the above results, we are ready to prove Theorem \ref{thm:asymp_full}.
\begin{proof}[{Proof of Theorem \ref{thm:asymp_full}}.]
Let $\kappa_t = 1 - \zeta_t$ denote the determinant \emph{discrepancy} between $R(\bar U)$ and $R(U)$. According to Lemma \ref{lem:exp_mono_zeta} we have the following:
\begin{align}
    \mathbb{E}\left[\frac{\kappa_{t + 1}}{\kappa_t}\bigg\lvert U\right] \leq 1 - \frac{1 - \kappa_t}{d} 
    \label{eq:eta_rate_0}
\end{align}
Now according to Lemma \ref{lem:mono_full},  $\kappa_t \leq 1 - \zeta_0$ for all $t \geq 0$. So using Eq (\ref{eq:eta_rate_0}) we have the following:
  \begin{equation}
    \mathbb{E}\left[\kappa_{t + 1}\big\lvert U\right]\leq \left(1 - \frac{1 - \kappa_t}{d}\right)\kappa_t \leq \left(1 - \frac{\zeta_0}{ d}\right)\kappa_t  \nn \;.
  \end{equation}
  Taking expectation of both sides, we have
  \begin{equation}
    \mathbb{E}\left[\kappa_{t + 1}\right] \leq \left(1 - \frac{\zeta_0}{ d}\right)\mathbb{E}\left[\kappa_t\right]  \nn \;.
  \end{equation}
  After $K \geq \frac{d}{\zeta_0}\log\frac{1}{\rho(1 - \zeta^{\ast})}\geq \frac{d}{\zeta_0}\log\frac{\mathbb{E}[\eta_{K_1}]}{\rho(1 - \zeta^{\ast})}$ iterations of GROUSE we obtain
  \begin{align}
    \mathbb{E}\left[\kappa_{t + K_1}\right] \leq \left(1 - \frac{\zeta_0}{d}\right)^{K} \mathbb{E}[\kappa_{0}] &\leq \left(1 - \frac{\zeta_0}{ d}\right)^{\frac{d}{\zeta_0}\log\frac{\mathbb{E}[\kappa_{0}]}{\rho(1 - \zeta^{\ast})}} \mathbb{E}[\kappa_{0}] \leq \rho(1 - \zeta^{\ast}) \nn \;.
  \end{align}
  Therefore
  \begin{align}
    \mathbb{P}\left(\zeta_{K} \geq \zeta^{\ast}\right) = 1 - \mathbb{P}\left(\kappa_{K} \geq 1 - \zeta^{\ast}\right) &\geq 1 - \frac{\mathbb{E}\left[\kappa_{K}\right]}{1 - \zeta^\ast} \geq 1 - \rho \;.
  \end{align}
\end{proof}

To get full convergence results, we need the following lemma, which gives us guarantees for a random initial point. 
\begin{lemma}\cite{nguyen2014random}
Initialize the starting point $U_0$ of GROUSE as the orthonormalization of an $n\times d$ matrix
with entries being standard normal random variables. Then 
\begin{equation}
    \mathbb{E}[\zeta_0] = \mathbb{E}\left[\det(U_0^T\bar U\bar U^TU_0)\right] = C \left(\frac{d}{n e}\right)^d \nonumber
    \label{eq:exp_zeta0}
 \end{equation} 
 where $C>0$ is a constant.    
  \label{lem:exp_zeta0}
\end{lemma}
}

\new{Now we will show a result that gives evidence for Conjecture \ref{thm:global}. 

\begin{theorem}[Global Convergence of GROUSE: Evidence for Conjecture \ref{thm:global}]
 Let $1 \geq \zeta^*>0$ be the desired accuracy of our estimated subspace. 
Let $\rho$ be any number within the range $(0,1]$. Let $\bar{\zeta_t}$ be a non-decreasing sequence with $\mathbb{E}[\bar{\zeta}_0] = \mathbb{E}[\zeta_0]$ such that
  \begin{equation}
    \mathbb{E}\left[\bar{\zeta}_{t + 1}\big\lvert U\right] \geq  \left(1 + \frac{\rho}{2 d}\right)\bar{\zeta}_t \;. \nn 
  \end{equation}
 Assume the $\zeta_t$ produced by GROUSE converges faster than $\bar{\zeta_t}$, \ie  
  \begin{align}
    \mathbb{E}\left[{\zeta}_{K_1}\right] \geq \mathbb{E}\left[\bar{\zeta}_{K_1}\right] \geq 1 - \frac{\rho}{2} 
    \label{eq:full_assumption}
  \end{align}
Suppose the initialization for GROUSE ($U_0$) is the range of an $n\times d$ matrix with entries being i.i.d standard normal random variables. Then after 
\begin{align}
K &\geq K_1 + K_2 \nn \\
&=\left(\frac{2 d^2}{\rho} + 1\right)\tau_0 \log (n) + 2 d \log \left( \frac{1}{2\rho(1 - \zeta^{\ast})}\right) \nn
\end{align} 
iterations of GROUSE Algorithm~\ref{alg:grouse},  
    \begin{equation}
 \mathbb{P}\left(\zeta_K \geq \zeta^\ast\right) \geq 1 - 2\rho\;, \nn
    \end{equation}
    where $\tau_0 = 1 + \frac{\log \frac{(1 - \rho/2)}{C} + d\log (e/d)}{d\log n}$ with $C$ a constant approximately equal to $1$. 
  \label{thm:global-tmp}
\end{theorem}}

\begin{proof}[{Proof of Theorem \ref{thm:global-tmp}}]
  Let $\kappa_t = 1 - \zeta_t$ denote the determinant \emph{discrepancy} between $R(\bar U)$ and $R(U)$. According to Lemma \ref{lem:exp_mono_zeta} we have the following:
  \begin{subequations}
    \begin{align}
      & \mathbb{E}\left[\frac{\zeta_{t+1}}{\zeta_t}\bigg\lvert U\right] \geq 1 + \frac{1 - \zeta_t}{d} \label{eq:zeta_rate} \\
      & \mathbb{E}\left[\frac{\kappa_{t + 1}}{\kappa_t}\bigg\lvert U\right] \leq 1 - \frac{1 - \kappa_t}{d} \label{eq:eta_rate}
    \end{align} 
  \end{subequations}
  Therefore, the expected convergence rate of $\zeta_t$ is faster when $R(U)$ is far away from $R(\bar U)$, while that of $\kappa_t$ is faster when $R(U)$ is close to $R(\bar U)$. This motivates us to split the analysis into two phases, bounding the number of iterations in each phase. We first use Eq (\ref{eq:zeta_rate}) to get the necessary $K_1$ iterations for GROUSE to converge to a local region of global optimal point from a random initialization. From there, we obtain the necessary  $K_2$ iterations for GROUSE to converge to the required accuracy by leveraging Eq (\ref{eq:eta_rate}).
  
  \new{As in the assumptions, let $\rho$ be any number within the range $(0,1]$. Let $\bar{\zeta_t}$ be a non-decreasing sequence with $\mathbb{E}[\bar{\zeta}_0] = \mathbb{E}[\zeta_0]$ and the expected increase rate being lower bounded as 
  \begin{equation}
    \mathbb{E}\left[\bar{\zeta}_{t + 1}\big\lvert U\right] \geq  \left(1 + \frac{\rho}{2 d}\right)\bar{\zeta}_t \;. \nn 
  \end{equation}
  Taking expectation of both sides, we obtain the following:
  \begin{equation}
    \mathbb{E}\left[\bar{\zeta}_{t+1}\right] \geq \left(1 + \frac{\rho}{2 d}\right)\mathbb{E}[\bar{\zeta}_t]   \nn
  \end{equation}
  Therefore after $K_1 \geq (2 d / \rho + 1)\log\frac{1 - \frac{\rho}{2}}{\mathbb{E}[\zeta_0]}$ steps we have 
  \begin{align}
    \mathbb{E}\left[\bar{\zeta}_{K_1}\right] \geq \left(1 + \frac{\rho}{2 d}\right)^{K_1} \mathbb{E}[\zeta_0] &\geq \left(\left(1 + \frac{\rho}{2 d}\right)^{\frac{2 d}{\rho} + 1}\right)^{\log\frac{1 - \frac{\rho}{2}}{\mathbb{E}[\zeta_0]}} \mathbb{E}[\zeta_0] \nn \\
    &\geq \mathbb{E}[\zeta_0] e^{\log\frac{1 - \frac{\rho}{2}}{\mathbb{E}[\zeta_0]}} = 1 - \frac{\rho}{2}  
  \end{align}
Now we apply the assumption in \eqref{eq:full_assumption}, that the $\zeta_t$ produced by GROUSE converges faster than $\bar{\zeta_t}$.
  Therefore, 
  \begin{align}
      \mathbb{P}\left({\zeta}_{K_1} \geq \frac{1}{2}\right) &= 1 - \mathbb{P}\left(1 - {\zeta}_{K_1} \geq \frac{1}{2}\right) \overset{\vartheta_1}\geq 1 - \frac{\mathbb{E}[1 - {\zeta}_{K_1}]}{1/2} \geq 1 - \rho  
      \label{pf:K1_prob}
  \end{align}
  where $\vartheta_1$ follows by applying Markov inequality to the nonnegative random variable $1 - \bar{\zeta}_{K_1}$.
  }

Now with probability at least $1 - \rho$, $\zeta_t \geq \frac{1}{2}$ for all $t \geq K_1$, \ie $\kappa_t \leq \frac{1}{2}$ for all $t \geq K_1$. So using Eq (\ref{eq:eta_rate}) we have the following:
  \begin{equation}
    \mathbb{E}\left[\kappa_{t + 1}\big\lvert U\right]\leq \left(1 - \frac{1 - \kappa_t}{d}\right)\kappa_t \leq \left(1 - \frac{1}{2 d}\right)\kappa_t  \nn \;.
  \end{equation}
  Taking expectation of both sides, we have
  \begin{equation}
    \mathbb{E}\left[\kappa_{t + 1}\right] \leq \left(1 - \frac{1}{2 d}\right)\mathbb{E}\left[\kappa_t\right]  \nn \;.
  \end{equation}
  After $K_2 \geq 2 d\log\frac{1/2}{\rho(1 - \zeta^{\ast})}\geq 2 d\log\frac{\mathbb{E}[\eta_{K_1}]}{\rho(1 - \zeta^{\ast})}$ additional iterations of GROUSE we obtain
  \begin{align}
    \mathbb{E}\left[\kappa_{t + K_1}\right] \leq \left(1 - \frac{1}{2 d}\right)^{K_2} \mathbb{E}[\kappa_{K_1}] &\leq \left(1 - \frac{1}{2 d}\right)^{2 d\log\frac{\mathbb{E}[\kappa_{K_1}]}{\rho(1 - \zeta^{\ast})}} \mathbb{E}[\kappa_{K_1}] \leq \rho(1 - \zeta^{\ast}) \nn \;.
  \end{align}
  Hence following a similar argument as before we have 
  \begin{align}
    \mathbb{P}\left(\zeta_{K_1 + K_2} \geq \zeta^{\ast}\right) = 1 - \mathbb{P}\left(\kappa_{K_1 + K_2} \geq 1 - \zeta^{\ast}\right) &\geq 1 - \frac{\mathbb{E}\left[\kappa_{K_1 + K_2}\right]}{1 - \zeta^\ast} \geq 1 - \rho \;.
    \label{pf:K2_bound}
  \end{align}
(\ref{pf:K1_prob}) and (\ref{pf:K2_bound}) together complete the proof.
\end{proof}
\new{Although we still need more rigorous analysis to justify our assumption, this proof provides the form of the convergence rate we can expect. We also want to emphasize that the above proof  provides the local convergence rate for GROUSE. Specifically, as is indicated by the proof of the second phase, GROUSE requires at most $2 d\log\frac{1/2}{\rho(1 - \zeta^{\ast})}$ iterations to converge from $\zeta_t = 1/2$ to any required accuracy $\zeta^\ast \in (1/2, 1)$.  }

\subsection{Proof of Undersampled Data}
\label{sec:proof_of_undersampled_data}
In this section, we prove our main results for undersampled data. 
We again start by proving a result for the deterministic expression for the change in determinant similarity from one step of the GROUSE algorithm to the next, in this case a lower bound given by Lemma \ref{lem:det_incr_expr_undersample}.
\begin{proof}[{Proof of Lemma \ref{lem:det_incr_expr_undersample}}]
Note that, 
\begin{subequations}
  \begin{align}
  &w^T(\bar U^TU)^{-1}\bar U^T p = w^T(\bar U^TU)^{-1}\bar U^T U w = \|p\|^2 \label{eq:norm_p} \\
  &w_1^T(\bar U^TU)^{-1}\bar U^T r \overset{\vartheta_1}= s^T\bar U^T U(\bar U^TU)^{-1}\bar U^Tr = v^T A^T\widetilde r \overset{\vartheta_2}= \|\widetilde r\|^2 \label{eq:norm_rtilde} 
\end{align}
\label{eq:rp_norm}
\end{subequations}
where $\vartheta_1$ follows by Lemma \ref{lem:uniq_expr} and $\vartheta_2$ holds since $v^T A^T\widetilde r=v^TA^T\left(\mathbb{I}_m - \mathcal{P}_{AU}\right)\widetilde r = \|\widetilde r\|^2$. As a consequence, we have the following 
\begin{align}
  \det \left(\bar U^T U_{t + 1}\right) &= \det\left(\bar U^T U + \bar U^T\left(\frac{p + r}{\|p + r\|} - \frac{p}{\|p\|}\right)\frac{w^T}{\|w\|}\right) \nn \\
    &\overset{\vartheta_3}= \det (\bar U^T U )  \frac{w^T(\bar U^T U)^{-1}\bar U^T\left(p + r\right)}{\|p\| \sqrt{\|p\|^2 + \|r\|^2}} \nn \\
    &= \det (\bar U^T U )  \frac{\|p\|^2 + \|r\|^2 + \|\widetilde r\|^2 - \|r\|^2 + \Delta}{\|p\| \sqrt{\|p\|^2 + \| r\|^2}} \nn 
\end{align}
 \sloppy{where $\Delta = w_2^T\left(\bar U^TU\right)^{-1}\bar U^T r$; and $\vartheta_3$ follows by the Schur complement $\det\left(M + ab^T\right) = \det(M)\left(1 + b^TM^{-1}a\right)$ for any invertible $M\in \R^{n\times n}$ and $a, b\in\R^{n}$. Hence}
 \begin{equation}
  \frac{\bar{\zeta}_{t+1}}{\zeta_t}= \left(\frac{\det\left(\bar U^T U_{t + 1}\right)}{\det\left(\bar U^TU\right)}\right)^2 \overset{\vartheta_4}\geq 1 + \frac{\|r\|^2}{\|p\|^2} + 2\frac{\left\|\widetilde r\right\| - \|r\|^2}{\|p\|^2} + 2\frac{\Delta}{\|p\|^2}  \nn
 \end{equation}
 where $\vartheta_4$ holds since $(c + d)^2 \geq c^2 + 2 cd$ with $c = \frac{\|p\|^2 + \|r\|^2}{\|p\|\sqrt{\|p\|^2 + \|r\|^2}}$, $d = \frac{\|\widetilde r\|^2 - \|r\|^2 + \Delta}{\|p\| \sqrt{\|p\|^2 + \| r\|^2}}$. 
 \end{proof}
In the following sections, we proceed by establishing the convergence results of missing data and compressively sampled data by bounding the key quantities in Lemma \ref{lem:det_incr_expr_undersample}.

\paragraph{Proof for Compressively Sampled Data} 
We start by showing how the results on the key quantities in Lemmas \ref{lem:rconc_cs}, \ref{lem:cs_rpconc} and \ref{lem:exp_delta_cs} lead to the main result of the compressively sampled data case. 
\begin{proof}[Proof of Theorem \ref{thm:det_convg_csrate}]
  Let $\eta_1 = \frac{1 + \delta}{1 - \delta}\frac{d}{m}$, $\eta_2 = (1 - \delta)\left(1 - 2\delta\sqrt{\frac{m}{n}}\right)$ and $\eta_3 = \tan(\phi_d) + \delta\frac{d}{\cos(\phi_d)}$,  then plugging in the results in Lemma \ref{lem:rconc_cs} to Lemma \ref{lem:exp_delta_cs} into Lemma \ref{lem:det_incr_expr_undersample} with $\delta_1=\delta_2=\delta_3=\delta$ yields,
  \begin{align}
      \frac{\zeta_{t + 1}}{\zeta_t} &\geq 1 + \frac{2\left\|\widetilde r\right\|^2 - \|r\|^2}{\|p\|^2} + 2\frac{\Delta}{\|p\|^2} \nn \\
      &\geq 1 + \frac{1}{\left(1 + \sqrt{\eta_1}\right)^2}\left(\eta_2(1 - \eta_1) - 2\sqrt{\eta_1}\eta_3\right)\frac{m}{n}\frac{\|v_\perp\|^2}{\|v\|^2} \nn \\
      &= 1 + \gamma_1\left(1 - \gamma_2\frac{d}{m}\right)\frac{m}{n}\frac{\|v_\perp\|^2}{\|v\|^2} \nn \\
      &= \left(1 + \gamma_1\left(1 - \gamma_2\frac{d}{m}\right)\frac{m}{n}\right)\frac{1 - \zeta_t}{d}
      \label{eq:prob_bnd_cs}
  \end{align}
  where {$\gamma_2 = \left(1 + 2\frac{\eta_3}{\eta_2\sqrt{\eta_1}}\right)\frac{1 + \delta}{1 - \delta} = \left(1 + 2\frac{\tan(\phi_d) + \delta_3\frac{d}{\cos(\phi_d)}}{\left(1 - 2\delta\sqrt{\frac{m}{n}}\right)\sqrt{(1 - \delta^2)d/m}}\right)\frac{1 + \delta}{1 - \delta}$, $\gamma_1 = \frac{\eta_2}{\left(1 + \sqrt{\eta_1}\right)^2} = \frac{(1 - \delta)\left(1 - 2\delta\sqrt{\frac{m}{n}}\right)}{\left(1 + \sqrt{\frac{1 + \delta}{1 - \delta}\frac{d}{m}}\right)^2}$,} and the last equality follows from Lemma \ref{lem:key_quantity}. 

  The probability bound is obtained by taking the union bound of those quantities (in Lemma \ref{lem:subrandmap}, Lemma \ref{lem:cs_randSub_proj}, Lemma \ref{thm:singular_val}, Corollary \ref{coro:subsp_union_bnd}, Lemma \ref{lem:Ubar_vperp})  used to generate Lemma \ref{lem:rconc_cs} to Lemma \ref{lem:exp_delta_cs}.  As we can see, this union bound is 
  {\small\begin{align}
      &1 - \exp\left(-\frac{m\delta^2}{2}\right) - \exp\left(-\frac{d\delta^2}{8}\right) - \exp\left(-\frac{m\delta^2}{32}+ d\log\left(\frac{24}{\delta}\right)\right) - (4d + 1)\exp\left(-\frac{m\delta^2}{8}\right) \nn \\
      &\qquad > 1 - \exp\left(-\frac{d\delta^2}{8}\right) - \exp\left(-\frac{m\delta^2}{32}+ d\log\left(\frac{24}{\delta}\right)\right) - (4d + 2)\exp\left(-\frac{m\delta^2}{8}\right) 
      \label{eq:cs_thm_prob_bnd}
    \end{align}}
To get the complexity bound on $m$, let $\varepsilon = \tan(\phi_d)$, $\alpha_1 = \varepsilon + \delta\sqrt{1 + \varepsilon^2}d$, $\alpha_2 = \frac{1 + \delta}{1 - \delta}$ and $\alpha_3 = \left(1 - 2\delta\sqrt{\frac{m}{n}}\right)\sqrt{1 + \delta}$, then according to \ref{eq:det_incr_cs_proof} we have $\gamma_2\frac{d}{m}<\frac{1}{2}$ is equivalent to the following,
    \begin{align}
      &\alpha_2 d + \frac{2\alpha_1\alpha_2\sqrt{d}}{\alpha_3}\sqrt{m} < \frac{m}{2} \nn \\
      &\Leftrightarrow \left(\sqrt{\frac{m}{2}} - \frac{\alpha_1\alpha_2\sqrt{d}}{\alpha_3}\right)^2 > \left(\alpha_2 + \frac{\alpha_1^2 \alpha_2^2}{\alpha_3^2}\right)d \nn \\
      &\overset{\vartheta_1}\Leftarrow m \geq 8 \frac{\alpha_1^2 \alpha_2^2}{\alpha_3^2}d + 4\sqrt{\alpha_2}\frac{\alpha_1\alpha_2}{\alpha_3}d \nn \\
      &\overset{\vartheta_2}\Leftarrow m \geq \beta \left(\varepsilon + \delta\sqrt{1 + \varepsilon^2}d\right)\left(\varepsilon + \delta\sqrt{1 + \varepsilon^2}d + \frac{1}{2} \right) d
      \label{m_bnd_1}
    \end{align}
    where $\vartheta_1$ follows from $\sqrt{\left(\alpha_2 + \frac{\alpha_1^2 \alpha_2^2}{\alpha_3^2}\right)d} < \sqrt{\alpha_2 d}  + \frac{\alpha_1 \alpha_2}{\alpha_3}\sqrt{d}$; and $\vartheta_2$ follows by $\beta =\frac{8(1+\delta)}{(1-\delta)^2\left(1 - 2\delta\right)^2}$.

    To establish another bound on $m$ we can see that $m\geq \frac{32}{\delta^2}\log\left(\frac{24 n^{2/d}}{\delta}\right)d$ implies the following, 
    \begin{align}
      &\exp\left(-\frac{m\delta^2}{32} + d\log\left(\frac{24}{\delta}\right)\right) \leq \exp(-\log{n^2}) = \frac{1}{n^2} \label{prob_bnd_1} \\
      &(4d + 2)\exp\left(-\frac{m\delta^2}{8}\right) \leq \frac{(4d+2)}{n^8}\left(\frac{\delta}{24}\right)^{4d} \rightarrow 0 \label{prob_bnd_2}
    \end{align}
    Eqs (\ref{prob_bnd_1}) and (\ref{prob_bnd_2}) complete the proof for the bound on $m$ and justify the simplification of the probability bound in Eq (\ref{eq:cs_thm_prob_bnd}).
\end{proof}

Next we are going to prove the intermediate lemmas in Section \ref{sec:compressively_sampled_data}, \ie bound the key quantities in Lemma \ref{lem:det_incr_expr_undersample}, for which we need the following concentration results. 
\begin{lemma}
  Let $A\in \R^{m\times n}$ with entries being i.i.d Gaussian random variables distributed as $\mathcal{N}(0, 1/n)$, $v\in \R^n$ is an vector. Then for any $\delta \in (0, 1)$, with probability at least $1 - 2 \exp^{-m\delta^2/8}$, we have
  \begin{align}
      &\mathbb{P}\left(\|A v\|_2^2 > (1 + \delta) \frac{m}{n}\|v\|_2^2\right) < \exp\left(-\frac{m\delta^2}{8}\right) \;, \nn \\
      &\mathbb{P}\left(\|A v\|_2^2 < (1 - \delta) \frac{m}{n}\|v\|_2^2\right) < \exp\left(-\frac{m\delta^2}{8}\right) \;. \nn
   \end{align} 
   \label{lem:subrandmap}
\end{lemma}
\begin{proof}
Note that $Av$ is a random vector with i.i.d entries distributing as $\mathcal{N}\left(0, \|v\|_2^2/n\right)$. Therefore, $\frac{n\left\|Av\right\|_2^2}{\|v\|_2^2}$ is a chi-squared distribution with $m$ degrees of freedom, which yields,
\begin{align}
  & \mathbb{P}\left[\frac{n\left\|Av\right\|_2^2}{m\|v\|_2^2} - 1 > \delta\right] < \exp\left(-m\delta^2/8\right) \nn \\
  &\mathbb{P}\left[\frac{n\left\|Av\right\|_2^2}{m\|v\|_2^2} - 1 < -\delta\right] < \exp\left(-m\delta^2/8\right) \nn
\end{align}
\end{proof}

\begin{lemma}
  Let $A\in \R^{m\times n}$ be a random matrix whose entries are independent and identically distributed Gaussian random variables with mean zero, and variance $\gamma$ . Let $z_1, z_2 \in \R^n$ such that $z_1 \perp z_2$, then $Az_1$ and $Az_2$ are independent of each other.
  \label{lem:gauss_induced_indep} 
\end{lemma}
\begin{proof}
   Let $a_i^T$ denote the $i^{th}$ row of $A$ and $M = Az_1z_2^TA^T$. Then we have 
   \begin{align}
     &\mathbb{E}[M]_{ii} = \mathbb{E}\left[a_i^Tz_1 z_1^Ta_i\right] = z_1^T \mathbb{E}[a_ia_i^T]z_2 = \gamma z_1^Tz_2 = 0 \nn \\
     &\mathbb{E}[M]_{ij} = \mathbb{E}\left[a_i^Tz_1 z_1^Ta_j\right] = z_1^T \mathbb{E}[a_ia_j^T]z_2 =  0 \nn
   \end{align}
   Therefore $Az_1$ and $Az_2$ are uncorrelated. This together with the fact that both $Az_1$ and $Az_2$ are Gaussian distributed random vectors imply that $Az_1$ and $Az_2$ are independent.
 \end{proof} 

\begin{lemma}[\cite{vershynin2010introduction}, Corollary 5.35]
  Let $A$ be an $n\times m$ matrix ($n \geq m$) whose entries are independent standard normal random variables. Then for every $h\geq 0$, with probability at least $1 - 2\exp\left(-  h^2/2\right)$ one has
  \begin{equation}
     \sqrt{n} - \sqrt{m} - h \leq \sigma_{\text{min}}(A) \leq \sigma_{\text{max}}(A) \leq \sqrt{n} + \sqrt{m} + h 
   \end{equation} 
where $\sigma_{\text{min}}, \sigma_{\text{max}}$ denote the smallest and largest singular values of $A$. 
\label{thm:singular_val}
\end{lemma}

With the above results, we are able to call out the following intermediate result to quantify $\left\|\mathcal{P}_{AU}(Av_\perp)\right\|_2^2$, which is a key quantity that will be used for proving  Lemmas \ref{lem:rconc_cs}, \ref{lem:cs_rpconc} and \ref{lem:exp_delta_cs}.
\begin{lemma}
  Let $A\in \R^{m\times n}$ with entries being i.i.d Gaussian random variables distributed as $\mathcal{N}(0, 1/n)$, then for any $\delta\in (0,1)$ we have 
  \begin{equation}
    \left\|\mathcal{P}_{AU} A v_\perp\right\|_2^2 \leq (1 + \delta)\frac{d}{n}\|v_\perp\|_2^2 \nn
  \end{equation}
  hold with probability at least $ 1 - \exp\left(-\frac{d\delta^2}{8}\right)$.
  \label{lem:cs_randSub_proj}
\end{lemma}
\begin{proof}
Note that $Av_\perp$ is a Gaussian random vector with i.i.d entries distributed as $\mathcal{N}\left(0, \|v_\perp\|_2^2/n\right)$, and $AU$ is a Gaussian random matrix with i.i.d entries distributed as $\mathcal{N}\left(0, 1/n\right)$. Then according to Lemma \ref{lem:gauss_induced_indep}, $AU$ and $Av_\perp$ are independent of each other.  Therefore, $y = \mathcal{P}_{AU}(Av_\perp)$ is the projection of $Av_\perp$ onto a independent random $d$-dimensional subspace. According to the rotation invariance property of $Av_\perp$, $\left\|\mathcal{P}_{AU}(Av_\perp)\right\|$ is equivalent to the length of projecting $Av_\perp$ onto its first $d$ coordinates. Hence,
\begin{align}
  \mathbb{P}\left(\left\|\mathcal{P}_{AU}(Av_\perp)\right\|_2^2 = \sum_{k = 1}^{d} y_k^2 \leq (1 + \delta)\frac{d}{n}\|v_\perp\|_2^2\right) \geq 1 - \exp\left(-\frac{d\delta^2}{8}\right)
\end{align}
Similar to the proof for Lemma \ref{lem:subrandmap}, here the probability bound is followed from the concentration bound for Chi-squared distribution with degree $d$. 
\end{proof}

Now we start by proving that Lemma \ref{lem:rconc_cs} follows directly from Lemma \ref{lem:subrandmap} and Lemma \ref{thm:singular_val}. 

\begin{proof}[{Proof of Lemma \ref{lem:rconc_cs}}]
According to Lemmas \ref{lem:subrandmap} and \ref{lem:cs_randSub_proj}, we have
  \begin{align}
    \|\widetilde r\|_2^2 = \|\left(\mathbb{I}_m - \mathcal{P}_{AU}\right)Av_\perp\|_2^2 
    &= \|Av_\perp\|_2^2 - \|\mathcal{P}_{AU}(Av_\perp)\|_2^2 \nn \\
    &\geq (1 - \delta_1)\frac{m}{n}\|v_\perp\|_2^2 - (1 + \delta_1)\frac{d}{n}\|v_\perp\|_2^2  \nn \\
    &= (1 - \delta_1)\left(1 - \frac{1 + \delta_1}{1 - \delta_1}\frac{d}{m}\right)\frac{m}{n}\|v_\perp\|_2^2
    \label{eq:csconc_bnd1}
  \end{align}
hold with probability at least $1 - \exp\left(-\frac{m\delta_1^2}{8}\right) - \exp\left(-\frac{d\delta_1^2}{8}\right)$. As for the second part of Lemma \ref{lem:rconc_cs}, we have 
  \begin{align}
    2\|\widetilde r\|_2^2 - \|r\|_2^2 = 2\|\widetilde r\|_2^2 - \|A^T \widetilde r\|_2^2 
    &\geq (2 - \sigma_{\text{max}}^2(A^T)) \|\widetilde r\|_2^2 \nn\\
    &\overset{\vartheta_1}\geq \left(1 - 2\delta_2\sqrt{\frac{m}{n}}\right)\|\widetilde r\|_2^2 \nn \\
    &\geq \left(1 - 2\delta_2\sqrt{\frac{m}{n}}\right)(1 - \delta_1)\left(1 - \frac{1 + \delta_1}{1 - \delta_1}\frac{d}{m}\right)\frac{m}{n}\|v_\perp\|_2^2
  \end{align}
here $\vartheta_1$ follows from Lemma \ref{thm:singular_val} with $A_{ij}\sim \mathcal{N}(0, 1/n)$ and $h = \delta\sqrt{m/n}$. The probability bound $1 - \exp\left(-\frac{m\delta_1^2}{8}\right) - \exp\left(-\frac{d\delta_1^2}{8}\right) - \exp\left(-\frac{m\delta_2^2}{2}\right)$ is obtained by taking the union bound over \ref{eq:csconc_bnd1} and $\vartheta_1$.  
\end{proof}

To prove  Lemma \ref{lem:cs_rpconc} and Lemma \ref{lem:exp_delta_cs}, we need the following extra results which are implied by Lemma \ref{lem:subrandmap}. The corresponding proofs are provided at the end of this section. 
\begin{corollary}
  Under the conditions of Lemma \ref{lem:subrandmap}, for $x, y \in \R^n$ and $\delta$, with probability exceeding $1 - 4e^{-m\delta^2/8}$ we have 
  \begin{equation}
     \frac{m}{n}\left(x^T y - \delta \|x\| \|y\|\right) \leq x^T A^T A y \leq \frac{m}{n}\left(x^T y + \delta \|x\| \|y\|\right) \nonumber
  \end{equation}
  \label{lem:innerProd}
\end{corollary}

\begin{corollary}
  Under the condition of Lemma \ref{lem:subrandmap}, for any vector $v\in R(U)$ we have
  \begin{align}
      &\mathbb{P}\left(\|A v\|_2^2 > (1 + \delta) \frac{m}{n}\|v\|_2^2\right) < \exp\left(-\frac{m\delta^2}{32} - d\log(\delta) + d\log(24) \right) \;, \nn \\
      &\mathbb{P}\left(\|A v\|_2^2 < (1 - \delta) \frac{m}{n}\|v\|_2^2\right) < \exp\left(-\frac{m\delta^2}{32}- d\log(\delta) + d\log(24)\right) \;. \nn
  \end{align}
  \label{coro:subsp_union_bnd}
\end{corollary}

Given Lemma \ref{lem:innerProd} and Corollary \ref{coro:subsp_union_bnd}, we prove Lemma \ref{lem:cs_rpconc} and Lemma \ref{lem:exp_delta_cs} by first proving the following intermediate results to bound the key components of $p$ and $\Delta$.
\begin{lemma}
Let $w_2 = \left(U^TA^TAU\right)^{-1}U^TA^TAv_\perp$, then 
  \begin{align}
    &\mathbb{P}\left(\|w_2\| \leq \sqrt{\frac{1 + \delta_1}{1 - \delta_2}\frac{d}{m}}\|v_\perp\|\right) \nn \\
    &\qquad \geq 1 - \exp\left(-\frac{d\delta_1^2}{8}\right) - \exp\left(-\frac{m\delta_2^2}{8}- d\log(\delta_2) + d\log(24)\right) \nn
  \end{align}
  \label{lem:w2_bnd}
\end{lemma}
\begin{proof}
Given the fact that $U\in \R^{n\times d}$ with columns being orthonormal, we have $\|w_2\| = \|Uw_2\|$. It then follows that,
  \begin{equation}
    \|U w_2\| \overset{\vartheta_1}\leq \frac{\|AUw_2\|}{\sqrt{(1 - \delta_2)m/n}} \overset{\vartheta_2}\leq \sqrt{\frac{1 + \delta_1}{1 - \delta_2}\frac{d}{m}}\|v_\perp\| \nn
  \end{equation}
  where $\vartheta_1$ follows from Corollary \ref{coro:subsp_union_bnd}, and $\vartheta_2$ followed by Lemma \ref{lem:cs_randSub_proj}, \ie 
  \begin{align}
    \left\|AUw_2\right\| = \left\|\mathcal{P}_{AU}(Av_\perp)\right\| \leq \sqrt{(1 + \delta_1)\frac{d}{n}\|v_\perp\|^2} \nn
  \end{align}
The probability bound is obtained by applying the union bound over $\vartheta_1$ and $\vartheta_2$.
\end{proof}

\begin{lemma}
Let $\phi_d$ denote the largest principal angle between $R(U)$ and $R(\bar U)$, then 
\begin{equation}
  \mathbb{P}\left(\left\|\bar{U}^TA^TAv_\perp\right\| \leq \left(\sin\phi_d + d\delta\right)\frac{m}{n}\|v_\perp\|\right) \geq 1 - 4d \exp\left({-\frac{m\delta^2}{8}}\right) \nn
\end{equation}
  \label{lem:Ubar_vperp_cs}
\end{lemma}
\begin{proof}[{Proof of Lemma \ref{lem:Ubar_vperp_cs}}]
Let $\bar u_k$ denote the $k^{th}$ column of $\bar U$, and $\delta \in (0,1)$. Then 
  \begin{align}
    \left\|\bar U^TA^TAv_\perp\right\| &=  \left\|\bar U^T\left(A^TA - \frac{m}{n}\mathbb{I}_n\right)v_\perp + \frac{m}{n}\bar U^T v_\perp \right\| \nn \\
    &\leq \frac{m}{n}\left\|\bar U^T v_\perp\right\| + \left\|\bar U^T\left(A^TA - \frac{m}{n}\mathbb{I}_n\right)v_\perp\right\| \nn \\
    &= \frac{m}{n}\left\|\bar U^T v_\perp\right\| + \sqrt{\sum_{k=1}^d \left(\bar u_k^TA^TAv_\perp - \frac{m}{n}\bar u_k^T v_\perp\right)^2} \nn \\
    &\overset{\vartheta_1}\leq \frac{m}{n}\left\|\bar U^T v_\perp\right\| + \sqrt{\sum_{k=1}^d \left(\delta\frac{m}{n}\|\bar u_k\| \|v_\perp\|\right)^2} \nn \\
   &\overset{\vartheta_2}\leq \sin\phi_{d}\frac{m}{n}\|v_\perp\| + \frac{m}{n}d\delta\|v_\perp\|
  \end{align}
  where $\vartheta_1$ follows from Lemma \ref{lem:innerProd}; $\vartheta_2$ holds from Lemma \ref{lem:Ubar_vperp} and the fact that $\sqrt{\sum_{k=1}^d \left(\delta\frac{m}{n}\|\bar u_k\| \|v_\perp\|\right)^2} \leq d\delta\frac{m}{n}\|\bar u_k\| \|v_\perp\| $; and the probability bound is obtained by taking the union bound of that in Lemma \ref{lem:innerProd}.
\end{proof}

We are ready to prove Lemma \ref{lem:cs_rpconc} and Lemma \ref{lem:exp_delta_cs}.
\begin{proof}[{Proof of Lemma \ref{lem:cs_rpconc}}]
Let $\eta = \sqrt{\frac{1 + \delta_1}{1 - \delta_1}\frac{d}{m}}$, then according to Lemma \ref{lem:w2_bnd} we have  
\begin{align*}
  \|p\|^2 = \| Uw_1 + Uw_2\|^2 &\leq \left(\|v_\parallel\| + \|Uw_2\|\right)^2 \nn \\
  &\leq \left(\|v_\parallel\| + \eta\|v_\perp\|\right)^2 \nn \\
  &\leq (1 + \eta)^2 \|v\|^2 \nn
\end{align*}
with probability at least $$1 - \exp\left(-\frac{m\delta_1^2}{32}- d\log(\delta_1) + d\log(24)\right) - \exp\left(-\frac{d\delta_1^2}{8}\right)\;.$$
Here the probability bound is obtained by choosing $\delta_1 = \delta_2$ in Lemma \ref{lem:w2_bnd}, hence completes the proof.
\end{proof}

\begin{proof}[{Proof of Lemma \ref{lem:exp_delta_cs}}]
According to the definition of $\Delta$, we can see Lemma \ref{lem:exp_delta_cs} is a direct results of Lemma \ref{lem:w2_bnd} and Lemma \ref{lem:Ubar_vperp}, that is
  \begin{align}
    \lvert\Delta\lvert &= w_2^T\left(\bar U^T U\right)^{-1} \bar U^TA^T \left(\mathbb{I}_m - \mathcal{P}_{AU}\right)Av_\perp \nn \\
    &\leq \left\|w_2^T\right\| \left\|\left(\bar U^T U\right)^{-1}\right\| \left\| \bar U^TA^T \left(\mathbb{I}_m - \mathcal{P}_{AU}\right)Av_\perp\right\| \nn \\
    &\overset{\vartheta_1}\leq \|w_2\| \left\|\left(\bar U^T U\right)^{-1}\right\| \left\|\bar U^TA^T Av_\perp\right\| \nn \\
    &\overset{\vartheta_2}\leq \frac{1}{\cos(\phi_d)}\sqrt{\frac{1 + \delta_1}{1 - \delta_1}\frac{d}{m}}\|v_\perp\| \left(\sin\phi_{d}\frac{m}{n} + \frac{m}{n}d\delta_3\right)\|v_\perp\| \nn \\
    &= \frac{1}{\cos(\phi_d)}\sqrt{\frac{1 + \delta_1}{1 - \delta_1}\frac{d}{m}}\left(\sin(\phi_d) + d\delta_3\right)\frac{m}{n}\|v_\perp\|^2 
    \label{eq:det_incr_cs_proof}
  \end{align}
  where $\vartheta_1$ holds since $\left\|\bar U^TA^T \left(\mathbb{I}_m - \mathcal{P}_{AU}\right)Av_\perp\right\| \leq \left\|\bar U^TA^TAv_\perp\right\|$; $\vartheta_2$ followed by Lemma \ref{lem:w2_bnd} and Lemma \ref{lem:Ubar_vperp_cs}; and the probability bound is obtained by taking the union bound that in Lemma \ref{lem:w2_bnd} and Lemma \ref{lem:Ubar_vperp_cs}.
\end{proof}

Finally, we are going to prove the auxiliary results Corollary \ref{coro:subsp_union_bnd} and Lemma \ref{lem:innerProd}. The key idea for proving Corollary \ref{coro:subsp_union_bnd} is using the covering numbers argument and applying Lemma \ref{lem:rconc_cs} to a given $d$-dimensional subspace $R(U)$. This is a common strategy used for compress sensing. 
\begin{proof}[{Proof of Corollary \ref{coro:subsp_union_bnd}}]
  Without loss of generality we restrict $\|v\| = 1$. From covering numbers \cite{szarek1997metric}, there exists a finite set $Q$ with at most $\left(\frac{24}{\delta}\right)^{d}$ points such that $Q \subset \R(U)$, $\|q\| = 1,  \forall q \in Q$, and for all $ x \in R(U)$ with $\|v\| = 1$ we can find a $q \in Q$ such that 
  \begin{equation}
    \|v - q\| \leq \delta / 8 \label{coverdist} \nonumber 
  \end{equation}
    Now applying Lemma \ref{lem:subrandmap} to the points in $Q$ with $\varepsilon = \delta /2$ and using the standard union bound, then with probability at least $1 - 2\left(\frac{24}{\delta}\right)^{d}\exp\left(-\frac{\delta^2}{32}m\right)$ we have 
    \begin{equation}
      (1 - \delta/2)\frac{m}{n}\|v\|^2 \leq \|A x\|^2 \leq (1 + \delta/2)\frac{m}{n}\|v\|^2 \nonumber
    \end{equation}
    which gives
    \begin{equation}
      \sqrt{1 - \delta/2}\sqrt{\frac{m}{n}}\|v\| \leq \|A x\| \leq \sqrt{1 + \delta/2}\sqrt{\frac{m}{n}}\|v\| 
    \end{equation}
    Since $\|v\| = 1$, we define $\gamma$ as the smallest number such that 
    \begin{equation}
      \|A x\| \leq \sqrt{1 + \gamma}\sqrt{\frac{m}{n}} \quad \forall x \in R(U) 
      \label{smallnumgoal}
    \end{equation}
    Since for any $x\in R(U)$ with $\|v\| = 1$ we can find a $q \in Q$ such that $\|x - q\| \leq \delta/8$, we have the following
    \begin{equation}
      \|Ax\| \leq \|A q\| + \|A(x - q)\| \leq \sqrt{1 + \delta/2}\sqrt{\frac{m}{n}} + \sqrt{1 + H}\sqrt{\frac{m}{n}}\delta/8 \nonumber
    \end{equation}
    Since $\gamma$ is the smallest number (\ref{smallnumgoal}) holds, we have $\sqrt{1 + \gamma} \leq \sqrt{1 + \delta/2} + \sqrt{1 + \gamma}\delta/8$.
    \begin{equation}
      \sqrt{1 + \gamma} \leq \frac{\sqrt{1 + \delta / 2}}{1 - \delta / 8} \leq \sqrt{1 + \delta}
    \end{equation} 
    Similarly, the lower bound follows by 
    \begin{eqnarray}
      \|A x\| \geq \|A q\| - \|A(x - q)\| &\geq& \sqrt{1 - \delta/2}\sqrt{\frac{m}{n}} - \sqrt{1 + \gamma}\frac{\delta}{8}\sqrt{\frac{m}{n}} \nonumber \\
      &\geq& \left(\sqrt{1 - \delta/2} - \sqrt{1 + \delta}\frac{\delta}{8}\right)\sqrt{\frac{m}{n}} \nonumber \\
      &\geq& \sqrt{1 - \delta}\sqrt{\frac{m}{n}} \nonumber
    \end{eqnarray}
  This completes the proof.
\end{proof}

\begin{proof}[{Proof of Lemma \ref{lem:innerProd}}]
Note that, 
  \begin{align}
    \frac{x^TA^TAy}{\|x\|\|y\|} &= \frac{1}{4}\left(\left\|A\left(\frac{x}{\|x\|} + \frac{y}{\|y\|}\right)\right\|^2 - \left\|A\left(\frac{x}{\|x\|} - \frac{y}{\|y\|}\right)\right\|^2\right) \nn 
  \end{align}
  Applying Lemma \ref{lem:subrandmap} on both terms separately and applying the union bound we have
  \begin{align}
  &\mathbb{P}\left[\frac{x^TA^TAy}{\|x\|\|y\|} \leq \frac{m}{n}\left(\frac{x^Ty}{\|x\|\|y\|} - \delta\right)\right] \nn \\
   &\qquad =\mathbb{P}\left[\frac{x^TA^TAy}{\|x\|\|y\|} \leq \frac{1}{4}\left((1 - \delta)\frac{m}{n}\left\|\frac{x}{\|x\|} + \frac{y}{\|y\|}\right\|^2 - (1 + \delta)\frac{m}{n}\left\|\frac{x}{\|x\|} - \frac{y}{\|y\|}\right\|^2\right)\right] \nn \\
   &\qquad < 2\exp\left(-\frac{m\delta^2}{8}\right) 
    \label{eq:inner_tail11}
  \end{align}
  Similarly, 
    \begin{align}
  &\mathbb{P}\left[\frac{x^TA^TAy}{\|x\|\|y\|} \geq \frac{m}{n}\left(\frac{x^Ty}{\|x\|\|y\|} + \delta\right)\right] \nn \\
   &\qquad =\mathbb{P}\left[\frac{x^TA^TAy}{\|x\|\|y\|} \geq \frac{1}{4}\left((1 + \delta)\frac{m}{n}\left\|\frac{x}{\|x\|} + \frac{y}{\|y\|}\right\|^2 - (1 - \delta)\frac{m}{n}\left\|\frac{x}{\|x\|} - \frac{y}{\|y\|}\right\|^2\right)\right] \nn \\
   &\qquad  < 2\exp\left(-\frac{m\delta^2}{8}\right)
    \label{eq:inner_tail22}
  \end{align}
  holds with probability no more than 
  \ref{eq:inner_tail11} and \ref{eq:inner_tail22} complete the proof.
\end{proof}

\paragraph{Proof of Missing Data} 

Here we again bound the quantities in Lemma \ref{lem:det_incr_expr_undersample}, Equation \ref{eq:det_ratio_undersample}, this time assuming $A$ represents an entry-wise observation operation and assuming incoherence on the signals of interest. As we show below, in the proof of Theorem \ref{thm:miss_convgrate}, we put together bounds given by Lemmas \ref{lem:rconc_miss}, \ref{lem:pconc_miss} and \ref{lem:delta_miss}, which are all proved in this section too, along with Lemma \ref{lem:uincoh} for completeness. 
 We start by proving the main result for missing data.
\begin{proof}[{Proof of ~Theorem \ref{thm:miss_convgrate}}]
Given the condition required by Theorem \ref{thm:miss_convgrate}, we have $\sin\phi_d \leq \sqrt{d\mu_0/16n}$ and $\cos\phi_d \geq \sqrt{1 - d\mu_0 / 16n}$. This together with Lemma \ref{lem:uincoh} and Lemma \ref{lem:delta_miss} yield $\left\lvert\Delta\right\lvert \leq \frac{\eta_3\sqrt{1 + \frac{m}{16n}}}{\sqrt{1 - d\mu_0 / 16n}}\frac{2d\mu_0}{n}\|v_\perp\|^2 \leq \frac{2\eta_3\sqrt{1 + \frac{1}{16}}}{\sqrt{1 - \frac{1}{16}}} \frac{d\mu_0}{n}\left\|v_\perp\right\|^2 \leq \frac{11}{5}\eta_3\frac{d\mu_0}{n}\left\|v_\perp\right\|^2$. Also for $\beta_2$ in Lemma \ref{lem:delta_miss} we have $\beta_2 \leq \sqrt{2\mu(v_\perp)\log(1/\delta)} = \beta_1$. Therefore, 
\begin{equation}
  \left\lvert\Delta\right\lvert \leq \frac{11}{5}\frac{(1 + \beta_1)^2}{1-\gamma_1}\frac{d\mu_0}{n}\|v_\perp\|^2 \;.
  \label{eq:simp_delta_miss}
\end{equation} 
Letting $\eta_2 = \frac{(1 + \beta_1)^2}{1-\gamma_1}\frac{d\mu_0}{m}$ and $\alpha_1 = \sqrt{\frac{2\mu(v_\perp)^2}{m}\log\left(\frac{1}{\delta}\right)}$, then applying this definition together with Lemma \ref{lem:uincoh} to Lemma \ref{lem:pconc_miss} Lemma \ref{lem:rconc_miss} yields
\begin{align}
  &\left\|p\right\|^2 \leq \left(1 + \sqrt{\frac{2\eta_2}{1 - \gamma_1}}\right)^2\|v\|^2 \label{eq:pmiss_simp} \\
  & \left\|r_\Omega\right\|^2 \geq (1 - \alpha_1 - 2\eta_2)\frac{m}{n}\left\|v_\perp\right\|^2 
  \label{eq:rmiss_simp}
\end{align}
Now applying \ref{eq:simp_delta_miss}, \ref{eq:pmiss_simp} and \ref{eq:rmiss_simp} to \ref{eq:miss_det} we obtain
\begin{align}
  \frac{\zeta_{t+1}}{\zeta_t} &\geq 1 + \frac{(1 - \alpha_1 - 2\eta_2)}{(1 + \sqrt{2\eta_2/(1 -\gamma_1)})^2}\frac{m}{n} \frac{\|v_\perp\|^2}{\|v\|^2} - \frac{22}{5}\frac{\eta_2}{(1 + \sqrt{2\eta_2/(1 -\gamma_1)})^2}\frac{m}{n}\frac{\|v_\perp\|^2}{\|v\|^2} \nn \\
  &\geq 1 + \frac{(1 - \alpha_1 - \frac{32}{5}\eta_2)}{(1 + \sqrt{2\eta_2/(1 -\gamma_1)})^2}\frac{m}{n} \frac{\|v_\perp\|^2}{\|v\|^2}
  \label{eq:miss_pf_det_raw}
\end{align}
which holds with probability at least $1 - 3\delta$. The probability bound is obtained by taking the union bound of those generating Lemmas \ref{lem:rconc_miss}, \ref{lem:pconc_miss} and \ref{lem:delta_miss}, as we can see in the proofs of them in this Section, this union bound is at least $1 - 3\delta$. 

Letting $\eta_1 = \frac{(1 - \alpha_1 - \frac{32}{5}\eta_2)}{(1 + \sqrt{2\eta_2/(1-\gamma_1)})^2}$, then $\eta_1> 0$ is equivalent to $1 - \alpha_1 - \frac{32}{5}\eta_2 >0$, for which we have the following: if 
{{\begin{equation}
    m > \max\left\{\frac{128d\mu_0}{3}\log\left(\frac{2d}{\delta}\right), 32 \mu(v_\perp)^2 \log\left(\frac{1}{\delta}\right), 52d\mu_0\left(1 + \sqrt{2\mu(v_\perp)\log\left(\frac{1}{\delta}\right)}\right)^2\right\}
    \label{eq:sample_complexity} 
  \end{equation}}}
then $\eta_1 > \frac{1}{4}$.

Under this condition, taking expectation with respect to $v$ yields, 
\begin{equation}
  \mathbb{E}_v\left[\frac{\zeta_{t + 1}}{\zeta_t} \big\lvert U\right] \geq 1 + \frac{1}{4}\frac{m}{n}\mathbb{E}\left[\frac{\|v_\perp\|^2}{\|v\|^2}\bigg\lvert U\right] \geq 1 + \frac{1}{4}\frac{m}{n}\frac{1 - \zeta_t}{d}
\end{equation}
where the last inequality follows from Lemma \ref{lem:key_quantity}. Finally choosing $\delta$ to be $1/n^2$completes the proof. 
\end{proof}

We then prove Corollary \ref{coro:discrep_decay}, the result that allows comparison between our convergence rate and that in \cite{balzano2014local}. 
\begin{proof}[{Proof of Corollary \ref{coro:discrep_decay}}]
Let $X = [X_1, \dots, X_d]$ with $X_i = \sin^2\phi_{i}$. Let $f(X) = 1 - \sum_{i = 1}^{d}X_i - \Pi_{i = 1}^{d}(1 - X_i)$, then $\frac{\partial f(X)}{\partial X_i} = -1 + \Pi_{j \neq i}(1 - X_j) \leq 0$. That is, $f(X)$ is a decreasing function of each component. Therefore, $f(X) \leq f(0) = 0$. It follows that 
\begin{equation}
  \zeta_t = \Pi_{i = 1}^{d}(1 - X_i) \geq 1 - \sum_{i = 1}^{d}X_i \geq 1 - \frac{d\mu_0}{16 n} 
  \label{eq:coro_key1}
\end{equation}
  With a slight modification of Theorem \ref{thm:miss_convgrate} we obtain
\begin{align}
  \mathbb{E}\left[\kappa_{t + 1} \big\lvert \kappa_t\right] \leq \left(1 - \frac{1}{4}\frac{m}{n}\frac{\zeta_t}{d}\right)\kappa_t  \;.
  \label{eq:coro_key2}
\end{align}
(\ref{eq:coro_key1}) and (\ref{eq:coro_key2}) together complete the proof. 
\end{proof}

We now focus on proving the key lemmas for establishing Theorem \ref{thm:miss_convgrate}, for which we need the following lemmas (the proofs can be found in \cite{balzano2010high}). 
\begin{lemma}\cite{balzano2010high}
  Let $\delta>0$. Suppose $m \geq \frac{8}{3}d\mu(U)\log\left(2d/\delta\right)$, then 
  \begin{equation}
    \mathbb{P}\left(\left\|\left(U_{\Omega}^TU_{\Omega}\right)^{-1}\right\| \leq \frac{n}{(1 - \gamma_1)m}\right) \geq 1 - \delta \nn
  \end{equation}
  where $\gamma_1 = \sqrt{\frac{8d\mu(U)}{3m}\log\left(2d/\delta\right)}$. 
  \label{lem:inv_Uomega}
\end{lemma}

\begin{lemma}[\cite{balzano2010high}, Lemma 1] Let $\alpha=\sqrt{\frac{2\mu(v_\perp)^2}{m}\log(1/\delta)}$, then 
\begin{equation}
  \mathbb{P}\left(\|v_{\perp, \Omega}\|^2 \geq (1 - \alpha)\frac{m}{n}\|v_\perp\|^2\right) \geq 1- \delta \nn
\end{equation}
\label{lem:Av_perp_miss_conc}
\end{lemma}

\begin{lemma}[\cite{balzano2010high}, Lemma 2]
Let $\mu(U), \mu(v_\perp)$ denote the incoherence parameters of $R(U)$ and $v_\perp$, and let $\delta \in (0,1)$ and $\beta_1 = \sqrt{2\mu(v_\perp)\log\left(1/\delta\right)}$, then 
  \begin{align}
    &\mathbb{P}\left(\left\|U_{\Omega}^T v_{\perp,\Omega}\right\|^2 \leq (\beta_1 + 1)^2\frac{m}{n}\frac{d\mu(U)}{n}\|v_{\perp}\|^2\right) \geq 1 - \delta \nn
  \end{align}
  \label{lem:U_v}
\end{lemma}

Now we are ready for the proof of Lemmas \ref{lem:rconc_miss}, \ref{lem:pconc_miss} and \ref{lem:delta_miss}.
\begin{proof}[{Proof of Lemma \ref{lem:rconc_miss}}]
According to Lemmas \ref{lem:Av_perp_miss_conc}, \ref{lem:U_v} and \ref{lem:inv_Uomega}, we have 
\begin{align}
  \left\|r_\Omega\right\|^2 &= \left\|v_{\perp,\Omega}\right\|^2 - v_{\perp,\Omega}^TU_\Omega\left(U_\Omega^TU_\Omega\right)^{-1}U_\Omega^T v_{\perp,\Omega} \nn \\
  &\geq\|v_{\perp,\Omega}\|^2 - \left\|\left(U_{\Omega}^TU_{\Omega}\right)^{-1}\right\|\| U_{\Omega}^T v_{\perp,\Omega}\|^2 \nn \\
  &\overset{\vartheta_1}\geq \left(1 - \alpha - \frac{(\beta_1+1)^2}{1-\gamma_1}\frac{d\mu(U)}{m}\right)\frac{m}{n}\|v_\perp\|^2 \nn
\end{align} 
with probability at least $1-3\delta$.
\end{proof}

\begin{proof}[{Proof of Lemma \ref{lem:pconc_miss}}]
Lemma \ref{lem:U_v} and Lemma \ref{lem:inv_Uomega} together give the following
\begin{align}
  \left\|Uw_2\right\|^2 = \left\|\left(U_\Omega^TU_\Omega\right)^{-1}U_\Omega^T v_{\perp,\Omega}\right\|^2 &\leq \left\|\left(U_\Omega^TU_\Omega\right)^{-1}\right\|^2 \left\|U_\Omega^T v_{\perp,\Omega}\right\|^2 \nn \\
  &\leq \frac{(\beta_1 + 1)^2}{(1 - \gamma_1)^2}\frac{d \mu(U)}{m}\|v_\perp\|^2 \nn
\end{align}
holds with probability exceeding $1 - 2\delta$. Therefore,
  \begin{equation}
    \|p\|^2 \leq \left(\|v_{\parallel}\| + \|Uw_2\|\right)^2 \leq \left(1 + \frac{\beta_1 + 1}{1 - \gamma_1}\sqrt{\frac{d\mu(U)}{m}}\right)^2\|v\|^2 \nn 
  \end{equation}
\end{proof}

We also need the following lemma for the proof of Lemma \ref{lem:delta_miss}, the proof of which is provided at the end of this section. 
\begin{lemma}
Let $\beta_2 = \sqrt{2\mu(v_\perp)\log\left(\frac{1}{\delta}\right)\frac{d\mu_0}{d\mu_0 + m\sin^2\phi_d}}$, where again $\mu_0$ denoting the incoherence parameter of $R(\bar U)$. Then  
\begin{equation}
  \mathbb{P}\left(\left\|\bar U_{\Omega}^T v_{\perp,\Omega}\right\|\leq (1 + \beta_2)\sqrt{\frac{m}{n}\frac{d\mu_0}{n}}\sqrt{\frac{m\sin^2\phi_d}{d\mu_0}+1}\|v_{\perp}\| \right) \geq 1 -\delta \nn 
\end{equation} 
  \label{lem:Ubar_v}
\end{lemma}

\begin{proof}[{Proof of Lemma \ref{lem:delta_miss}}]
Note that $\left\lvert\Delta\right\lvert = \left\|\Delta\right\|$, for which we have the following,
\begin{align}
  \left\|\Delta\right\| &= \left\|w_2^T(\bar U^T U)^{-1}\bar U^T r\right\| \nn \\
  &= \left\|v_{\perp,\Omega}^TU_\Omega\left(U_\Omega^TU_\Omega\right)^{-1}\left(\bar U^T U\right)^{-1}\bar U_{\Omega}^T \left(I - \mathcal{P}_{U_\Omega}\right) v_{\perp, \Omega}\right\| \nonumber \\
  &\leq \left\|v_{\perp,\Omega}^TU_\Omega\right\| \left\|\left(U_\Omega^TU_\Omega\right)^{-1}\right\| \left\|\left(\bar U^T U\right)^{-1}\right\| \left\|\bar U_{\Omega}^T \left(I - \mathcal{P}_{U_\Omega}\right) v_{\perp, \Omega}\right\| \nn \\
  &\overset{\vartheta_1}\leq \frac{1}{\cos\phi_{d}}\left\|v_{\perp,\Omega}^TU_\Omega\right\| \left\|\left(U_\Omega^TU_\Omega\right)^{-1}\right\|\left\|\bar U_{\Omega}^T v_{\perp, \Omega}\right\| \nn \\
  &\leq \frac{1}{\cos\phi_{d}}(\beta_1 + 1)\sqrt{\frac{m}{n}\frac{d\mu(U)}{n}}(1 + \beta_2)\sqrt{\frac{m}{n}\frac{d\mu_0}{n}}\sqrt{\frac{m\sin^2\phi_d}{d\mu_0}+1}\frac{n}{m(1-\gamma_1)}\|v_\perp\|^2 \nn \\
  &\overset{\vartheta_2}\leq \frac{(1+\beta_1)(1+\beta_2)}{(1-\gamma_1)\cos\phi_{d}}\sqrt{\frac{m\sin^2\phi_d}{d\mu_0}+1}\sqrt{\frac{d\mu_0}{n}}\sqrt{\frac{d\mu(U)}{n}}\|v_{\perp}\|^2 \nn
\end{align}
  where $\vartheta_1$ holds since from the following: $$\left\|\bar U_{\Omega}^T \left(I - \mathcal{P}_{U_\Omega}\right) v_{\perp, \Omega}\right\| \leq \left\|\bar U_{\Omega}^T v_{\perp, \Omega}\right\|\;, \quad\quad \left\|\left(U_\Omega^TU_\Omega\right)^{-1}\right\| \leq \frac{1}{\cos\phi_{d}}$$
  and $\vartheta_2$ follows by putting Lemmas \ref{lem:U_v}, \ref{lem:inv_Uomega} and \ref{lem:Ubar_v} together.
\end{proof}

We also prove Lemma \ref{lem:uincoh} for completeness. Before that we first call out the following lemma, the proof of which can be found in \cite{balzano2014local}. 
\begin{lemma}\cite{balzano2014local}
There exists an orthogonal matrix $V\in \R^{d\times d}$ such that 
\begin{equation}
  \sum_{k = 1}^{d}\sin^2\phi_{k} \leq \left\|\bar U V - U\right\|_F^2 \leq 2\sum_{k = 1}^{d}\sin^2\phi_{k} \nn 
 \end{equation} 
  \label{lem:Udist}
\end{lemma}

\begin{proof}[{Proof of Lemma \ref{lem:uincoh}}]
    According to Lemma \ref{lem:Udist} we have 
    \begin{align}
      \left\|U_{i}\right\|_2 \leq \left\|\bar U_{i}\right\|_2 + \left\|\bar U_{i} V - U_{i}\right\|_2 &\leq \left\|\bar U_{i}\right\| + \sqrt{2\sum_{k = 1}^{d}\sin^2\phi_k} \nn \\
      &\leq \left(1 + \frac{1}{2\sqrt{2}}\right)\sqrt{\frac{d\mu_0}{n}} \nn 
    \end{align}
    It hence follows that $\left\|U_{i}\right\|_2^2 \leq 2\frac{d \mu_0}{n}$.
\end{proof}

We need the following lemma and McDiarmid's inequality to prove Lemma \ref{lem:Ubar_vperp}.
\begin{lemma}
  $\left\|\bar U^T v_\perp\right\|^2 \leq \sin^2({\phi_d})\|v_\perp\|^2$, where $\phi_d$ denotes the largest principal angle between $R(\bar U)$ and $R(U)$.
  \label{lem:Ubar_vperp}
\end{lemma}
\begin{proof}
According to the definition of $v_\perp$ and Lemma \ref{lem:subsprelate}, we have 
  \begin{align}
  \left\|\bar U^T y\right\|^2 = \left\|\bar U^T\left(\mathbb{I} - UU^T\right)\bar U s\right\|^2 
  &= s^T\bar Y\Sigma^4\bar Y^4 s \nn \\
  &\overset{\vartheta_3}\leq \sin^2\phi_{d} s^T\bar Y\Sigma^2\bar Y^T s = \sin^2\phi_{d}\|v_{\perp}\|^2 \nn 
\end{align}
here $\bar Y$ and $\Sigma$ are the same as those defined in Lemma \ref{lem:subsprelate}, and the last equality holds since $\|v_{\perp}\|^2 = \|s\|^2 - v^TUU^Tv = s^T\bar Y\Sigma^2\bar Y^T s$. 
\end{proof}
\begin{theorem}{(McDiarmid's Inequality \cite{mcdiarmid1989method}).} Let $X_1,\dots,X_n$ be independent random variables, and assume $f$ is a function for which there exist $t_i$, $i = 1,\dots,n$ satisfying 
\begin{align}
  \sup_{x_1,\dots,x_n,\widehat x_i}\left\lvert f(x_1,\dots,x_n) - f(x_1,\dots,\widehat x_i, \dots, x_n)\right\lvert \leq t_i \nn
\end{align}
where $\widehat x_i$ indicates replacing the sample value $x_i$ with any other of its possible values. Call $f(X_1,\dots,X_n) := Y$. Then for any $\epsilon > 0$, 
\begin{align}
  &\mathbb{P}\left[Y \geq \mathbb{E}Y + \epsilon\right] \leq \exp\left(- \frac{2\epsilon^2}{\sum_{i = 1}^{n}t_i^2}\right) \nn \\
  &\mathbb{P}\left[Y \leq \mathbb{E}Y - \epsilon\right] \leq \exp\left(- \frac{2\epsilon^2}{\sum_{i = 1}^{n}t_i^2}\right) \nn
\end{align} 
\label{thm:McDiarmid}
\end{theorem}

\begin{proof}[{Proof of Lemma \ref{lem:Ubar_v}}]
\sloppy{We use McDiarmid's inequality to prove this. For the simplicity of notation denote $v_\perp$ as $y$. Let $X_i = \bar U_{\Omega(i)} y_{\Omega(i)}\in \R^{d}$, and $f(X_1,\dots,X_m) = \left\|\sum_{i = 1}^{m}X_i\right\|_2 = \left\|\bar U_{\Omega}^T v_{\perp,\Omega}\right\|_2$, then $\left\lvert f(x_1, \dots, x_n) - f(x_1, \dots,\widehat x_i, \dots x_n\right\lvert$ can be bounded via}
\begin{align}
  \left\lvert \left\|\sum_{i = 1}^{m}X_i\right\|_2 - \left\|\sum_{i \neq k}^{m}X_i + \widehat X_k\right\|_2\right\lvert \leq \left\|X_k - \widehat X_k\right\|_2 &\leq \|X_k\|_2 + \|\widehat X_k\|_2 \nn\\
  &\leq 2\|y\|_{\infty} \sqrt{d\mu_0/n}
  \label{eq:Mcdiarmid_abs}
\end{align}

We next calculate $\mathbb{E}\left[f(X_1,\dots,X_m)\right] = \mathbb{E}\left[\left\|\sum_{i = 1}^{m}X_i\right\|_2\right]$. Note that 
\begin{equation}
  \mathbb{E}\left[\left\|\sum_{i = 1}^{m}X_i\right\|^2\right] 
  = \mathbb{E}\left[\sum_{i=1}^{m}\|X_i\|^2 + \sum_{i = 1}^m\sum_{j\neq i}X_i^T X_j\right] 
  \label{eq:McdiarMid_exp_whole}
\end{equation}
Recall that we assume the samples are taken uniformly with replacement. This together with the fact that $\left\|\bar U_{i}\right\|_2 = \|\mathcal{P}_{R(\bar U)}(e_i)\| \leq \sqrt{d \mu_0 / n}$ yield the following 
\begin{align}
  \mathbb{E}\left[\sum_{i=1}^{m}\|X_i\|^2\right] &= \sum_{i=1}^{m}\mathbb{E}\left[\left\|U_{\Omega(i)}y_{\Omega_{(i)}}\right\|^2\right]\nn \\
  &=\sum_{i = 1}^{m}\sum_{k = 1}^{n}\|\bar U_{k}\|^2 y_{k}^2 \mathbb{P}_{\{\Omega(i) = k\}}\leq\frac{m}{n}\frac{d\mu_0}{n}\|y\|^2 
  \label{eq:McdiarMid_exp_squa}
\end{align}
\begin{align}
  \mathbb{E}\left[\sum_{i = 1}^m\sum_{j\neq i}X_i^T X_j\right] &= \sum_{i=1}^{m}\sum_{j\neq i}\sum_{k_1=1}^{n}\sum_{k_2=1}^{n}y_{k_1}\bar U_{k_1}^T\bar U_{k_2}y_{k_2}\mathbb{P}(\Omega_{j}=k_2)\mathbb{P}(\Omega_{i}=k_1) \nn \\
  &= \frac{m^2 - m}{n^2} \|\bar U^T y\|^2 \leq \frac{m^2}{n^2}\sin^2\phi_d \|y\|^2 
  \label{eq:McdiarMid_exp_cross}
\end{align}
where the last inequality holds by Lemma \ref{lem:Ubar_vperp}.

Eqs (\ref{eq:McdiarMid_exp_whole})  (\ref{eq:McdiarMid_exp_squa}) and (\ref{eq:McdiarMid_exp_cross}) together with the Jensen's inequality imply 
\begin{equation}
  \mathbb{E}\left[\left\|\sum_{i = 1}^{m}X_i\right\|\right] \leq \sqrt{\frac{m}{n}}\sqrt{\frac{m}{n}\sin^2\phi_d + \frac{d\mu_0}{n}}\|y\| = \sqrt{\frac{m}{n}\frac{d\mu_0}{n}}\sqrt{\frac{m\sin^2\phi_d}{d\mu_0}+1}\|y\| 
  \label{eq:mcd_exp_bnd}
\end{equation} 
Let $\epsilon = \beta_2\sqrt{\frac{m}{n}\frac{d\mu_0}{n}}\sqrt{\frac{m\sin^2\phi_d}{d\mu_0}+1}\|y\|$, then (\ref{eq:Mcdiarmid_abs}) and (\ref{eq:mcd_exp_bnd}) together with Theorem \ref{thm:McDiarmid} give
\begin{align}
  &\mathbb{P}\left[\left\|U_\Omega y_\Omega\right\| \geq (1 + \beta_2)\sqrt{\frac{m}{n}\frac{d\mu_0}{n}}\sqrt{\frac{m\sin^2\phi_d}{d\mu_0}+1}\|y\| \right] \nn \\
  &\qquad \qquad\leq \exp\left(\frac{-2\beta_2^2\frac{m}{n}\frac{d\mu_0}{n}\left(\frac{m\sin^2\phi_d}{d\mu_0}+1\right)\|y\|^2}{4m\|y\|_{\infty}^2 \frac{d\mu_0}{n}}\right) \nn \\
  &\qquad \qquad= \exp\left(\frac{-\beta_2^2\left(\frac{m\sin^2\phi_d}{d\mu_0}+1\right)\left\|y\right\|^2}{2n\left\|y\right\|_\infty^2}\right) = \delta
\end{align}
where the last inequality follows by submitting our definition of $\mu(y)$ Eq (\ref{defn:vec_incoh}) and $\beta_2$. 
\end{proof}

\bibliographystyle{plain}
\bibliography{IEEEgrouse}

\begin{thebibliography}{10}

\bibitem{absil2009optimization}
P-A Absil, Robert Mahony, and Rodolphe Sepulchre.
\newblock {\em Optimization algorithms on matrix manifolds}.
\newblock Princeton University Press, 2009.

\bibitem{armentano2014average}
Diego Armentano, Carlos Beltr{\'a}n, and Michael Shub.
\newblock Average polynomial time for eigenvector computations.
\newblock {\em arXiv preprint arXiv:1410.2179}, 2014.

\bibitem{arora2013stochastic}
Raman Arora, Andy Cotter, and Nati Srebro.
\newblock Stochastic optimization of {PCA} with capped {MSG}.
\newblock In {\em Advances in Neural Information Processing Systems}, pages
  1815--1823, 2013.

\bibitem{balsubramani2013fast}
Akshay Balsubramani, Sanjoy Dasgupta, and Yoav Freund.
\newblock The fast convergence of incremental {PCA}.
\newblock In {\em Advances in Neural Information Processing Systems}, pages
  3174--3182, 2013.

\bibitem{balzano2012handling}
Laura Balzano.
\newblock {\em Handling missing data in high-dimensional subspace modeling}.
\newblock PhD thesis, University of Wisconsin -- Madison, 2012.

\bibitem{balzano2022equivalence}
Laura Balzano.
\newblock On the equivalence of {Oja's} algorithm and {GROUSE}.
\newblock In {\em Proceedings of AIStats}, 2022.

\bibitem{balzano2010online}
Laura Balzano, Robert Nowak, and Benjamin Recht.
\newblock Online identification and tracking of subspaces from highly
  incomplete information.
\newblock In {\em 48th Annual Allerton Conference on Communication, Control,
  and Computing}, pages 704--711. IEEE, 2010.

\bibitem{balzano2010high}
Laura Balzano, Benjamin Recht, and Robert Nowak.
\newblock High-dimensional matched subspace detection when data are missing.
\newblock In {\em 2010 IEEE International Symposium on Information Theory},
  pages 1638--1642. IEEE, 2010.

\bibitem{balzano2014local}
Laura Balzano and Stephen~J Wright.
\newblock Local convergence of an algorithm for subspace identification from
  partial data.
\newblock {\em Foundations of Computational Mathematics}, pages 1--36, 2014.

\bibitem{bertsekas2011incremental}
Dimitri~P Bertsekas.
\newblock Incremental gradient, subgradient, and proximal methods for convex
  optimization: A survey.
\newblock {\em Optimization for Machine Learning}, 2010(1-38):3, 2011.

\bibitem{bhojanapalli2016dropping}
Srinadh Bhojanapalli, Anastasios Kyrillidis, and Sujay Sanghavi.
\newblock Dropping convexity for faster semi-definite optimization.
\newblock In {\em Conference on Learning Theory}, pages 530--582. PMLR, 2016.

\bibitem{brooks2013pure}
J~Paul Brooks, JH~Dul{\'a}, and Edward~L Boone.
\newblock A pure l1-norm principal component analysis.
\newblock {\em Computational statistics \& data analysis}, 61:83--98, 2013.

\bibitem{candes2011robust}
Emmanuel~J Cand{\`e}s, Xiaodong Li, Yi~Ma, and John Wright.
\newblock Robust principal component analysis?
\newblock {\em Journal of the ACM (JACM)}, 58(3):11, 2011.

\bibitem{chen2015fast}
Yudong Chen and Martin~J Wainwright.
\newblock Fast low-rank estimation by projected gradient descent: General
  statistical and algorithmic guarantees.
\newblock {\em arXiv preprint arXiv:1509.03025}, 2015.

\bibitem{d2008optimal}
Alexandre d'Aspremont, Francis Bach, and Laurent~El Ghaoui.
\newblock Optimal solutions for sparse principal component analysis.
\newblock {\em The Journal of Machine Learning Research}, 9:1269--1294, 2008.

\bibitem{sa2015global}
Christopher~D De~Sa, Christopher Re, and Kunle Olukotun.
\newblock Global convergence of stochastic gradient descent for some non-convex
  matrix problems.
\newblock In {\em Proceedings of the 32nd International Conference on Machine
  Learning (ICML-15)}, pages 2332--2341, 2015.

\bibitem{edelman1998geometry}
Alan Edelman, Tom{\'a}s~A Arias, and Steven~T Smith.
\newblock The geometry of algorithms with orthogonality constraints.
\newblock {\em SIAM journal on Matrix Analysis and Applications},
  20(2):303--353, 1998.

\bibitem{golub2012matrix}
Gene~H Golub and Charles~F Van~Loan.
\newblock {\em Matrix computations}.
\newblock JHU Press, 4 edition, 2012.

\bibitem{he2012incremental}
Jun He, Laura Balzano, and Arthur Szlam.
\newblock Incremental gradient on the grassmannian for online foreground and
  background separation in subsampled video.
\newblock In {\em Computer Vision and Pattern Recognition (CVPR), 2012 IEEE
  Conference on}, pages 1568--1575. IEEE, 2012.

\bibitem{jain2016streaming}
Prateek Jain, Chi Jin, Sham~M Kakade, Praneeth Netrapalli, and Aaron Sidford.
\newblock Streaming pca: Matching matrix bernstein and near optimal finite
  sample guarantees for oja's algorithm.
\newblock In {\em 29th Annual Conference on Learning Theory}, pages 1147--1164,
  2016.

\bibitem{jain2013low}
Prateek Jain, Praneeth Netrapalli, and Sujay Sanghavi.
\newblock Low-rank matrix completion using alternating minimization.
\newblock In {\em Proceedings of the forty-fifth annual ACM symposium on Theory
  of computing}, pages 665--674. ACM, 2013.

\bibitem{keshavan2010matrix}
Raghunandan~H Keshavan, Andrea Montanari, and Sewoong Oh.
\newblock Matrix completion from a few entries.
\newblock {\em Information Theory, IEEE Transactions on}, 56(6):2980--2998,
  2010.

\bibitem{mcdiarmid1989method}
Colin McDiarmid.
\newblock On the method of bounded differences.
\newblock {\em Surveys in combinatorics}, 141(1):148--188, 1989.

\bibitem{ngo2012scaled}
Thanh Ngo and Yousef Saad.
\newblock Scaled gradients on grassmann manifolds for matrix completion.
\newblock In {\em Advances in Neural Information Processing Systems}, pages
  1412--1420, 2012.

\bibitem{nguyen2014random}
Hoi~H Nguyen, Van Vu, et~al.
\newblock Random matrices: Law of the determinant.
\newblock {\em The Annals of Probability}, 42(1):146--167, 2014.

\bibitem{recht2010guaranteed}
Benjamin Recht, Maryam Fazel, and Pablo~A Parrilo.
\newblock Guaranteed minimum-rank solutions of linear matrix equations via
  nuclear norm minimization.
\newblock {\em SIAM review}, 52(3):471--501, 2010.

\bibitem{RH2012}
R.H.Keshavan.
\newblock {\em Efficient algorithms for collaborative filtering}.
\newblock PhD thesis, Stanford University, 2012.

\bibitem{stewart1990matrix}
Gilbert~W Stewart and Ji-guang Sun.
\newblock {\em Matrix perturbation theory}.
\newblock Academic press, 1990.

\bibitem{szarek1997metric}
Stanislaw~J Szarek.
\newblock Metric entropy of homogeneous spaces.
\newblock {\em arXiv preprint math/9701213}, 1997.

\bibitem{vershynin2010introduction}
Roman Vershynin.
\newblock Introduction to the non-asymptotic analysis of random matrices.
\newblock {\em arXiv preprint arXiv:1011.3027}, 2010.

\bibitem{xu2010robust}
Huan Xu, Constantine Caramanis, and Sujay Sanghavi.
\newblock Robust {PCA} via outlier pursuit.
\newblock In {\em Advances in Neural Information Processing Systems}, pages
  2496--2504, 2010.

\bibitem{yang1995projection}
Bin Yang.
\newblock Projection approximation subspace tracking.
\newblock {\em IEEE Transactions on Signal processing}, 43(1):95--107, 1995.

\bibitem{zhang2019extracting}
Dejiao Zhang.
\newblock {\em Extracting Compact Knowledge From Massive Data}.
\newblock PhD thesis, University of Michigan, Ann Arbor, 2019.

\bibitem{zhang2015global}
Dejiao Zhang and Laura Balzano.
\newblock Global convergence of a grassmannian gradient descent algorithm for
  subspace estimation.
\newblock In {\em AISTATS}, pages 1460--1468, 2016.

\bibitem{zheng2015convergent}
Qinqing Zheng and John Lafferty.
\newblock A convergent gradient descent algorithm for rank minimization and
  semidefinite programming from random linear measurements.
\newblock In {\em Advances in Neural Information Processing Systems}, pages
  109--117, 2015.

\end{thebibliography}

\end{document}